\documentclass[11pt,reqno]{amsart}
\usepackage{hyperref}
\usepackage{ amssymb }
\usepackage{mathrsfs}
\numberwithin{equation}{section}
\usepackage{float}
\usepackage{pgf,tikz}

\newcommand \reg{\operatorname{reg}}

\newcommand\Supp{\operatorname{Supp}}

\newcommand \C{\mathcal{C}}

\newcommand \G{\mathcal{G}}

\newcommand \K{\mathbb{K}}

\newcommand \cp{\operatorname{CP}}
\newcommand \Sim{\operatorname{Sim}}
\newcommand \pd{\operatorname{pd}}

\newcommand \con{\mathcal{C}}

\newcommand \ind{\operatorname{Ind}}

\newcommand{\dd}{\operatorname{d}}
\newtheorem{theorem}{Theorem}[section]
\newtheorem{definition}[theorem]{Definition}

\newtheorem{lemma}[theorem]{Lemma}

\newtheorem{example}[theorem]{Example}
\newtheorem{obs}[theorem]{Observation}
\newtheorem{question}[theorem]{Question}
\newtheorem{remark}[theorem]{Remark}

\newtheorem{corollary}[theorem]{Corollary}

\newtheorem*{notation*}{Notation}

\begin{document}

\title[Linear Resolutions of Connected Graph Ideals and Their Powers]
{Linear Resolutions of Connected Graph Ideals and Their Powers}

\author{Arka Ghosh}
\email{arkaghosh1208@gmail.com}

\author{S Selvaraja}
\email{selvas@iitbbs.ac.in}
\address{Department of Mathematics, Indian Institute Of Technology  Bhubaneswar, Bhubaneswar, 752050, India}

\subjclass[2020]{13F55, 05E45, 13D02}

\keywords{linear resolution, regularity, co-chordal clutters, connected ideals, powers of ideals}

\begin{abstract}
For a finite simple graph $G$ and an integer $r \ge 1$, the $r$-connected ideal $I_r(G)$ is the squarefree monomial ideal generated by the vertex sets of connected induced subgraphs of size $r+1$, extending the classical edge ideal. We investigate the linearity of the minimal free resolutions of $I_r(G)$ via structural features of the associated clutter $\mathcal{C}_r(G)$.
We introduce the class of co-chordal-cactus graphs and prove that $I_r(G)$ has a linear resolution for all $r \ge 2$ whenever $G$ lies in this family. The result further extends to $(2K_2, C_4)$-free graphs and co-grid graphs. For $r=1$, we show that the edge ideal $I_1(G)$ has Castelnuovo--Mumford regularity at most $3$ for all co-chordal-cactus and co-grid graphs.
We also examine powers of connected ideals and establish that $I_r(G)^q$ has a linear resolution for every $q \ge 1$ in several natural graph families, including complements of trees with bounded degree, complete multipartite graphs, complements of cycles, graphs obtained by gluing complete graphs along cliques, and certain subclasses of split graphs.
\end{abstract}

\maketitle

\section{Introduction}

A central theme in combinatorial commutative algebra is the study of the interplay between algebraic properties of squarefree monomial ideals and the combinatorics of the discrete structures that generate them, such as simplicial complexes, clutters, and graphs. Among the most extensively investigated homological invariants is the existence of a \emph{linear resolution}. If $I\subseteq \K[x_1,\dots,x_n]$ is a homogeneous ideal generated in a single degree and all entries of the maps in its minimal free resolution are either zero or linear forms, we say that $I$ has a \emph{linear resolution}. Characterizing squarefree monomial ideals that admit such resolutions has led to a substantial body of work.
A landmark result in this direction is due to Fr\"oberg~\cite{froberg}, who completely resolved the quadratic case. For a simple graph $G$ on $V(G)=\{x_1,\dots,x_n\}$, the \emph{edge ideal} is
\[
I(G)=\big( x_i x_j \mid \{x_i,x_j\}\in E(G)\big)\subseteq \K[x_1,\dots,x_n].
\]
A graph $H$ is \emph{chordal} if every induced cycle of length at least four has a chord, and \emph{co-chordal} if its complement $\overline{H}$ is chordal.  
Fr\"oberg’s theorem states that $I(G)$ has a linear resolution if and only if $\overline{G}$ is chordal~\cite[Theorem~1]{froberg}; equivalently, $I(G)$ is linear precisely when $G$ is co-chordal. This foundational result has inspired extensive efforts to generalize the characterization to higher-degree squarefree monomial ideals and richer combinatorial settings (see, e.g.,~\cite{BYZ17,CF13,CF15,Anton21,eagon,KK06,Russ11}).

A natural higher-dimensional analogue of the edge ideal is the \emph{connected ideal}. For a graph $G$ and integer $r\ge1$, define
\[
I_r(G)=\Big(\prod_{x\in S}x \ \Big|\ |S|=r+1,\ G[S]\text{ connected}\Big),
\]
which recovers $I(G)$ when $r=1$.  
Combinatorially, $I_r(G)$ is the Stanley--Reisner ideal of the \emph{$r$-independence complex} $\ind_r(G)$, whose faces are the subsets $A\subseteq V(G)$ such that each connected component of $G[A]$ has at most $r$ vertices (recovering the usual independence complex when $r=1$). Equivalently, $I_r(G)$ is the edge ideal of the clutter $\C_r(G)$ whose edges are the connected $(r+1)$-subsets of $V(G)$.

The algebraic behavior of connected ideals has recently received considerable attention~\cite{FPSAA23,AJM24,HJ15,KRK25,DRSV24,DochEng09,GS25,TG06}, particularly regarding linear resolutions. A major result of~\cite{DRSV24} shows that if $G$ is co-chordal (that is, $\overline{G}$ is chordal), then $I_r(G)$ has a linear resolution for all $r\ge1$. Their proof combines two independent approaches: vertex-splittable ideal techniques and topological arguments via collapsibility of $\ind_r(G)$. Additional sufficient conditions for linearity are known; for example,~\cite{AJM24} proved that $I_r(G)$ has a linear resolution whenever $G$ is both gap-free and $r$-claw-free, using linear quotients.

A powerful general framework for establishing linearity of squarefree monomial ideals was developed by Bigdeli, Yazdanpour, and Zaare-Nahandi~\cite{BYZ17}, who related this property to the notion of \emph{co-chordality} in clutters. A $d$-uniform clutter $\mathcal{C}$ is called \emph{chordal} if it can be reduced to the empty clutter by successively deleting simplicial maximal subedges, and \emph{co-chordal} if its complement $\overline{\mathcal{C}}$ is chordal. Their fundamental theorem~\cite[Theorem~3.3]{BYZ17} states that if $\mathcal{C}$ is co-chordal, then the edge ideal $I(\mathcal{C})$ has a linear resolution.  
Importantly, although linearity of squarefree monomial ideals may in general depend on the characteristic of the ground field~\cite{Re76}, the result of~\cite{BYZ17} holds over \emph{every} field~$\K$; thus linearity in the co-chordal setting is characteristic-independent.

Applied to our setting, this shows that to establish a linear resolution for the connected ideal $I_r(G)$, it is enough to verify that the associated clutter $\mathcal{C}_r(G)$ is co-chordal. Determining co-chordality of a given clutter, however, is typically a highly nontrivial combinatorial problem. This difficulty motivates one of the central goals of the present work. Although recent results~\cite{DRS25} have identified certain graph classes for which $\mathcal{C}_r(G)$ is co-chordal for all $r\ge2$-often through connections with the Cohen–Macaulayness of the $r$-co-connected complex of $G$-a general combinatorial characterization remains unknown. The main challenge is to convert the recursive nature of clutter co-chordality into concrete, verifiable graph-theoretic conditions on $G$.

In this paper we address this challenge by developing new combinatorial techniques that identify broad families of graphs whose higher connected ideals admit linear resolutions. Our approach extends the co-chordal paradigm to more general graph structures while preserving strong homological behavior. More precisely, we focus on three natural classes that generalize classical co-chordal phenomena:
\begin{enumerate}
    \item[(i)] \emph{co-chordal-cactus graphs}, whose complements are obtained by gluing cycles and chordal graphs along a cactus skeleton; this family properly contains all co-chordal, co-cycle, and co-cactus graphs;
    \item[(ii)] \emph{co-grid graphs}, i.e., complements of grid graphs (Cartesian products of paths);
    \item[(iii)] \emph{$(2K_2,C_4)$-free graphs}, which forbid both $2K_2$ and $C_4$ as induced subgraphs.
\end{enumerate}

Our first main result shows that each of the graph classes introduced above guarantees co-chordality of the connected clutter.

\begin{theorem}[Theorems~\ref{thm:cycle-cochordal}, \ref{2k2-c4-free}, \ref{thm:con-cochordal}]
\label{intro-co-chordal}
Let $G$ belong to one of the following graph families:
\begin{enumerate}
    \item co-chordal-cactus graphs,
    \item $(2K_2,C_4)$-free graphs,
    \item co-grid graphs.
\end{enumerate}
Then, for every $r\ge2$, the clutter $\mathcal{C}_r(G)$ is co-chordal.  
Consequently, the ideal $I_r(G)$ has a linear resolution; equivalently,
\[
\reg(I_r(G)) = r+1 \qquad \text{for all } r\ge2,
\]
where $\reg(-)$ denotes the Castelnuovo–Mumford regularity.
\end{theorem}

As a direct consequence, we recover several known results (see Corollary~\ref{cor:DRSV24}).  
For instance, it is known~\cite{Nursel} that if $G$ is $(2K_2,C_4)$-free, then $\reg(I_1(G))\le3$.  
The next theorem shows that this regularity bound holds for two much larger graph families.

\begin{theorem}[Theorems~\ref{thm:reg-co-cactus}, \ref{reg-grid}]
\label{intro-reg}
If $G$ is a co-chordal-cactus graph or a co-grid graph, then
\[
\reg(I_1(G)) \le 3.
\]
\end{theorem}

Together, Theorems~\ref{intro-co-chordal} and~\ref{intro-reg} reveal a striking stabilization phenomenon:  
for the graph classes considered here, the edge ideal $I_1(G)$ need not be linear, yet its regularity remains uniformly bounded by~$3$, while every higher connected ideal $I_r(G)$ with $r\ge2$ always admits a linear resolution.
Beyond the connected ideals $I_r(G)$, we also study linearity for edge ideals associated with the complement clutters.  
In particular, we prove that for two fundamental graph families—block graphs and cycles-the ideals
$I\bigl(\,\overline{\con_r(G)}\,\bigr)$
have linear resolutions for all $r\ge2$ (Theorems~\ref{block-p4} and~\ref{cycle-chordal}).  
This may be viewed as a higher-dimensional analogue of the case $r=1$, where Fröberg’s theorem characterizes precisely when $I(\overline{\con_1(G)})$ is linear.

We now turn to the behavior of powers of ideals that admit linear resolutions.  
It is well known that linearity is not in general preserved under taking powers.  
A classical counterexample, due to Sturmfels~\cite{St00}, is the squarefree monomial ideal
$I=(def,\, cef,\, cdf,\, cde,\, bef,\, bcd,\, acf,\, ade),$
which has a linear resolution, whereas $I^{2}$ does not.  
In sharp contrast, edge ideals exhibit a strong persistence phenomenon:  
Herzog, Hibi, and Zheng~\cite{HHZ} proved that if the edge ideal $I_{1}(G)$ of a graph $G$ has a linear resolution-equivalently, if $G$ is co-chordal-then every power $I_{1}(G)^{q}$ also has a linear resolution for all $q\ge1$.  
This naturally raises the question
\[
\text{If } I_{r}(G)\ \text{has a linear resolution, must}\ I_{r}(G)^{q}\ \text{have a linear resolution for all } q\ge1\ ? 
\]

We provide substantial evidence for this question by identifying broad families of graphs for which all powers of $I_r(G)$ retain linear resolutions.

\begin{theorem}[Theorems~\ref{tree-complement}, \ref{thm:linear-resolution-cases}, \ref{partially-split}]
\label{thm:powers-linear}
Let $G$ belong to one of the following graph families:
\begin{enumerate}
    \item the complement of a tree $T$ with maximum degree $\Delta(T)\le r$;
    \item complete multipartite graphs;
    \item the complement of a cycle;
    \item graphs $\Gamma_{p,m_1,\dots,m_n}$ obtained by gluing complete graphs along a clique of size at most $r$;
    \item partially split graphs.
\end{enumerate}
Then $I_r(G)^q$ has a linear resolution for every $q\ge1$.
\end{theorem}

Theorem~\ref{thm:powers-linear} unifies and substantially generalizes all previously known instances of persistence of linear resolutions, extending the Herzog-Hibi-Zheng theorem from the case $r=1$ to a wide array of higher connected ideals.  
In particular, it identifies several new and natural graph families for which every power of $I_r(G)$ is linear.

\medskip

\noindent\textit{Organization of the paper.}  
Section~\ref{preliminaries} introduces notation and background from commutative algebra, graph theory, and clutter theory.  
Section~\ref{TL} develops the technical lemmas used throughout the sequel.  
Section~\ref{main} establishes that co-chordal-cactus graphs, $(2K_2,C_4)$-free graphs, and co-grid graphs all satisfy that $\mathcal{C}_r(G)$ is co-chordal for every $r\ge2$.  
Section~\ref{reg-ls} provides regularity bounds and linearity results for the corresponding ideals.  
Finally, Section~\ref{main1} proves Theorem~\ref{thm:powers-linear}, establishing linearity of all powers $I_r(G)^q$ for the graph families listed above.

\section{Preliminaries}\label{preliminaries}
In this section, we establish the basic definitions and notation required for the main results.
\subsection{Algebraic Background}

Let $R=\K[x_1,\dots,x_n]$ be the standard graded polynomial ring over a field $\K$, and let $I\subseteq R$ be a homogeneous ideal.  
The homological structure of $I$ is encoded by its minimal graded free resolution
\[
0 \longrightarrow F_p \longrightarrow F_{p-1} \longrightarrow \cdots 
\longrightarrow F_1 \longrightarrow F_0 \longrightarrow I \longrightarrow 0,
\]
where 
\[
F_i=\bigoplus_{j\in\mathbb{Z}} R(-j)^{\beta_{i,j}(I)},
\qquad
R(-j)_d = R_{d-j}.
\]
Here $p=\pd(I)\le n$ is the projective dimension and the integers $\beta_{i,j}(I)$ are the graded Betti numbers.  
The \emph{Castelnuovo--Mumford regularity} of $I$ is
\[
\reg(I)=\max\{\, j-i \mid \beta_{i,j}(I)\ne0 \,\},
\qquad 
\text{so that }\reg(R/I)=\reg(I)-1.
\]

If $I=(f_1,\dots,f_t)$ is a homogeneous ideal generated in a single degree $d$, then $I$ is said to have a \emph{linear resolution} if its minimal free resolution is $d$-linear, i.e.,
\[
0\longrightarrow R(-d-p)^{\beta_{p,p+d}(I)}
   \longrightarrow \cdots 
   \longrightarrow R(-d-1)^{\beta_{1,d+1}(I)}
   \longrightarrow R(-d)^{\beta_{0,d}(I)}
   \longrightarrow I \longrightarrow 0,
\]
equivalently $\beta_{i,j}(I)=0$ whenever $j\neq i+d$.  
In particular, $\reg(I)=d$.
A useful sufficient condition for linearity is the existence of \emph{linear quotients}.  
The ideal $I$ has linear quotients if its minimal generators admit an order
$f_1<\cdots<f_t$ such that
\[
(f_1,\dots,f_{k-1}):(f_k)
\quad\text{is generated by linear forms for all } k\ge2.
\]
If, moreover, all generators have the same degree, then $I$ has a linear resolution by \cite[Proposition~8.2.1]{Herzog'sBook}.

\subsection{Graph Theory Background}

Let $G$ be a finite simple graph without isolated vertices.  
We write $V(G)$ and $E(G)$ for its vertex and edge sets.  
For $x\in V(G)$, the \emph{degree} is 
$\deg_G(x)=|\{\,y\in V(G)\mid \{x,y\}\in E(G)\,\}|.$
A subgraph $H\subseteq G$ is \emph{induced} if  
$\{u,v\}\in E(H)\iff\{u,v\}\in E(G)$ for all $u,v\in V(H)$.  
For $A\subseteq V(G)$, the induced subgraph is
$G[A]=(A,\ \{\{u,v\}\in E(G)\mid u,v\in A\}).$
For $U\subseteq V(G)$, deletion means $G\setminus U=G[V(G)\setminus U]$.
For vertices $u_1,\dots,u_r\in V(G)$, the open and closed neighborhoods are
\[
N_G(u_1,\dots,u_r)=\{v\in V(G)\mid \{u_i,v\}\in E(G)\text{ for some $i$}\},
\]
\[
N_G[u_1,\dots,u_r]=N_G(u_1,\dots,u_r)\cup\{u_1,\dots,u_r\}.
\]

A path of length $k-1$ is a sequence of distinct vertices 
$P_k:v_1v_2\cdots v_k$ with $\{v_i,v_{i+1}\}\in E(G)$.  
A cycle of length $n\ge3$ is 
$C_n:v_1v_2\cdots v_nv_1$.  
The graph $G$ is \emph{connected} if every pair of vertices is joined by a path; a \emph{connected component} is a maximal connected induced subgraph.  
The \emph{distance} between vertices $u,v$ is 
\[
\dd_G(u,v)=\text{length of the shortest $u$--$v$ path}.
\]
A vertex $v$ is a \emph{cut vertex} if $G\setminus\{v\}$ has more connected components than $G$.
A graph is \emph{$2K_2$-free} if it contains no induced copy of two disjoint edges, i.e., no induced matching of size~$2$.  
A vertex $v$ is \emph{simplicial} if $G[N_G(v)]$ is a clique.  
If $G_1,\dots,G_n$ have pairwise disjoint vertex sets, their \emph{join} is
$\bigvee_{i=1}^n G_i,$
defined by
\[
V\!\left(\bigvee_{i=1}^n G_i\right)=\bigcup_{i=1}^n V(G_i),
\]
\[
E\!\left(\bigvee_{i=1}^n G_i\right)
=\Big(\bigcup_{i=1}^n E(G_i)\Big)
\ \cup\
\{\{u,v\}\mid u\in V(G_i),\,v\in V(G_j),\,i\ne j\}.
\]

\subsection{Clutters and Edge Ideals}

A \emph{clutter} $\mathcal{C}$ on a finite vertex set $V(\mathcal{C})$ is a family of subsets $E(\mathcal{C})$ (called \emph{edges}) such that no edge contains another.  
If all edges have the same cardinality $d$, then $\mathcal{C}$ is a \emph{$d$-uniform clutter} (or \emph{$d$-clutter}).  
A $2$-clutter is precisely a simple graph.
Let $V(\mathcal{C})=\{x_1,\dots,x_n\}$, and identify each $x_i$ with a variable of $\K[x_1,\dots,x_n]$.  
The \emph{edge ideal} of $\mathcal{C}$ is the squarefree monomial ideal
\[
I(\mathcal{C})=\Big(\prod_{x\in e} x \ \Bigm|\ e\in E(\mathcal{C})\Big),
\]
generalizing classical graph edge ideals and providing a uniform framework for squarefree monomial ideals generated in one degree~\cite{ha_adam}.  
For a $d$-uniform clutter $\mathcal{C}$, the \emph{complement clutter} $\overline{\mathcal{C}}$ is defined on the same vertex set by
\[
E(\overline{\mathcal{C}})=\{\,e\subseteq V(\mathcal{C})\mid |e|=d,\ e\notin E(\mathcal{C})\,\}.
\]
Two clutters $\mathcal{C}_1,\mathcal{C}_2$ are \emph{isomorphic}, written $\mathcal{C}_1\cong\mathcal{C}_2$, if there exists a bijection 
$\varphi:V(\mathcal{C}_1)\to V(\mathcal{C}_2)$ such that
\[
e\in E(\mathcal{C}_1)\ \Longleftrightarrow\ \varphi(e)\in E(\mathcal{C}_2),
\qquad 
\varphi(e)=\{\varphi(x):x\in e\}.
\]

We recall the terminology used to define chordality of clutters following~\cite{BYZ17}.  
A subset $W\subseteq V(\mathcal{C})$ is a \emph{clique} if every $d$-subset of $W$ is an edge; by convention, all subsets of size $<d$ are cliques.  
For a $(d{-}1)$-subset $\sigma\subseteq V(\mathcal{C})$, its \emph{closed neighborhood} is
\[
N_{\mathcal{C}}[\sigma]=\sigma\cup\{\,v\in V(\mathcal{C})\mid \sigma\cup\{v\}\in E(\mathcal{C})\,\}.
\]
If $N_{\mathcal{C}}[\sigma]\neq\sigma$, then $\sigma$ is a \emph{maximal subedge}.  
A maximal subedge $\sigma$ is \emph{simplicial} if $N_{\mathcal{C}}[\sigma]$ is a clique.  
Let $\Sim(\mathcal{C})$ denote the set of simplicial maximal subedges.
For $\sigma\subseteq V(\mathcal{C})$, the \emph{deletion} of $\sigma$ from $\mathcal{C}$ is
\[
\mathcal{C}\setminus\sigma
=\{\,e\in E(\mathcal{C}) \mid \sigma\not\subseteq e\,\}.
\]
For multiple deletions, we use the shorthand
$\mathcal{C}\setminus\sigma_1\setminus\cdots\setminus\sigma_t
\ =\
\mathcal{C}\setminus\{\sigma_1,\dots,\sigma_t\}.$

A $d$-uniform clutter $\mathcal{C}$ is \emph{chordal} if either $E(\mathcal{C})=\emptyset$, or there exists $\sigma\in\Sim(\mathcal{C})$ such that $\mathcal{C}\setminus\sigma$ is chordal.  
Equivalently, $\mathcal{C}$ can be reduced to the empty clutter by repeatedly deleting simplicial maximal subedges.  
A $d$-uniform clutter $\mathcal{C}$ is \emph{co-chordal} if $\overline{\mathcal{C}}$ is chordal.

\medskip

For a graph $G$ and an integer $r\ge1$, the \emph{$(r{+}1)$-clutter} of $G$ is
\[
V(\mathcal{C}_r(G))=V(G),\qquad
E(\mathcal{C}_r(G))
=\{\,S\subseteq V(G)\mid |S|=r+1,\ G[S]\ \text{connected}\,\}.
\]
For $r=1$, $\mathcal{C}_1(G)$ coincides with the edge set of $G$.  
The associated connected ideal is the squarefree monomial ideal
\[
I_r(G)=I(\mathcal{C}_r(G))
=\Big(\prod_{x\in S}x \ \Bigm|\ S\subseteq V(G),\ |S|=r+1,\ G[S]\text{ connected}\Big).
\]
In particular, $I_1(G)$ is the usual edge ideal.

\section{Technical Lemmas}\label{TL}

This section is devoted to auxiliary results that will be used in the proofs of the main theorems.  
The lemmas and propositions presented here form the technical backbone for the arguments developed in the subsequent sections.
We begin with a simple but useful observation regarding complete graphs.  
Since every induced subgraph of a complete graph is itself complete, the corresponding $(r+1)$-clutters admit an immediate description of chordality.  
This is recorded in the following lemma.

\begin{lemma}\label{compl-chordal}
Let $K_n$ be the complete graph on $n$ vertices. If $r\ge 1$ and $r+1\le n$, then the clutter $\mathcal{C}_r(K_n)$ is chordal.
\end{lemma}

\begin{proof}
We proceed by induction on $n$.
Suppose $n=r+1$.  
Let $V(K_n)=\{1,\dots,r+1\}$ and $e=\{1,\dots,r\}$. Then
$N_{\mathcal{C}_r(K_n)}[e]=e\cup\{r+1\}=V(K_n).$
Since $V(K_n)$ itself is an edge of $\mathcal{C}_r(K_n)$, it follows that 
\(e\in \mathrm{Sim}(\mathcal{C}_r(K_n))\), and
$\mathcal{C}_r(K_n)\setminus e=\emptyset.$
Thus $\mathcal{C}_r(K_n)$ is chordal.

Assume that $\mathcal{C}_r(K_n)$ is chordal for some $n\ge r+1$.  
Let $V(K_{n+1})=\{1,\dots,n+1\}$ and define
\[
\mathcal{E}
=
\{\, e\subset V(K_{n+1}) : |e|=r,\ 1\in e,\ n+1\notin e \,\}.
\]
Order $\mathcal{E}$ lexicographically:
\[
\mathcal{E}=\{e_1,e_2,\dots,e_m\},
\qquad
e_i=\{a_1<\cdots<a_r\},\ a_1=1,\ a_r\le n.
\]
For $i\ge 1$, set
$D_i=\mathcal{C}_r(K_{n+1})\setminus\{e_1,\dots,e_{i-1}\}.$

\medskip
\noindent
\textsc{Claim.}
$e_i\in \operatorname{Sim}(D_i)$ for all $i=1,\dots,m$.

\medskip
\noindent
\emph{Proof of claim.}
Observe that,
$N_{D_i}[e_i]
=
e_i\cup\{\,v\in V(K_{n+1}) : v>a_r\,\},$
and $a_r\le n$. Now suppose $e_j\subseteq e_i\cup \{n+1\}$ for some $1\le j\le i-1$. Since $n+1\notin e_j$, therefore $e_j\subseteq e_i$ and since $|e_j|=|e_i|$, therefore $e_j=e_i$, which is a contradiction.  
So $n+1\in N_{D_i}[e_i]$. Hence $ N_{D_i}[e_i]\not= D_i$.
For $j<i$, $e_j$ contains some vertex $b_j<a_r$ that is not in $N_{D_i}[e_i]$. Thus
$e_j\not\subset N_{D_i}[e_i].$
Consequently, every $(r+1)$-subset of $N_{D_i}[e_i]$ is an edge of $D_i$, so $e_i$ is simplicial in $D_i$.  
\hfill $\square$

\medskip

Deleting $e_1,\dots,e_m$ from $\mathcal{C}_r(K_{n+1})$ produces a clutter isomorphic to
$\mathcal{C}_r(K_{n+1}\setminus\{1\}) \cong \mathcal{C}_r(K_n).$
By the induction hypothesis, $\mathcal{C}_r(K_n)$ is chordal, and since
$(e_1,\dots,e_m)$ is a simplicial elimination sequence, the original clutter
$\mathcal{C}_r(K_{n+1})$ is chordal as well.
\end{proof}

A \emph{block graph} is a graph in which every block-that is, every maximal $2$-connected subgraph-is a clique.
The next result demonstrates that, for this class of graphs, the property of admitting a linear resolution extends naturally to the higher connected ideals associated to their complements.

\begin{theorem}\label{block-p4}
If $G$ is a block graph, then the ideal $I(\overline{\con_r(G)})$ has a linear resolution for all $r\ge 2$.
\end{theorem}

\begin{proof}
We apply induction on $|V(G)|$. If $|V(G)|=2$, then $\mathcal{C}_r(G)$ is trivially chordal.
Assume the result holds for all block graphs with at most $n$ vertices, and let $G$ be a block graph on $n+1$ vertices. Let $\cp(B_1,\dots,B_m)$ be a maximal clique path with connecting vertices $x_1,\dots,x_{m-1}$ (see \cite[Page~7]{GS25}). Choose $x_0\in V(B_1)\setminus\{x_1\}$.
Set
\[
\mathcal{B}=\{\, e\in E(\mathcal{C}_r(G)) : x_0\in e \,\},\qquad
A_e=\{\, z\in e : \dd_G(x_0,z)\ge2\,\},\qquad
A=\bigcup_{e\in\mathcal{B}} A_e,
\]
and let $g_1,\dots,g_s\in E(\mathcal{C}_r(G))$ be those edges for which $x_0\in g_i$ and $g_i\cap A\neq\emptyset$. Then $x_1\in g_i$ for all $i$, and we define $e_i=g_i\setminus\{x_1\}$.

\medskip\noindent\textsc{Claim.}
 $e_i\in\operatorname{Sim}(\con_r(G)\setminus \{e_1, \dots, e_{i-1}\})$ for $i=1,\dots,s$.

\smallskip\noindent
Since $N_{\mathcal{C}_r(G)}[e_1]=e_1\cup\{x_1\}=g_1\in E(\mathcal{C}_r(G))$, then $e_1\in \operatorname{Sim}(\con_r(G))$.  For $i>1$, set $D=\mathcal{C}_r(G)\setminus\{e_1,\dots,e_{i-1}\}$. Because $e_j\neq e_i$ for $j<i$, we have $e_j\not\subseteq e_i\cup\{x_1\}$; hence $g_i=e_i\cup\{x_1\}\in E(D)$ and $N_D[e_i]=g_i$, showing that $e_i$ is maximal simplicial in $D$. \hfill$\square$

\medskip
Now let $x_0 \in e$, where $e \in E\!\left(\con_r(G) \setminus \{e_1,\dots,e_s\}\right)$.  
Observe that $e \subseteq V(B_1)$.  
Arguing as in Lemma~\ref{compl-chordal}, we obtain $e_{s+1},\dots,e_t$ such that
\[
e_i \in \Sim\!\left(\overline{\con_r(G)} \setminus \{e_1,\dots,e_{i-1}\}\right)
\qquad\text{for } i=s+1,\dots,t,
\]
and
$\overline{\con_r(G)} \setminus \{e_1,\dots,e_t\}
= \con_r\!\bigl(G \setminus \{x_0\}\bigr).$
The induction hypothesis then completes the argument.

Finally, by \cite[Theorem~3.3]{BYZ17}, chordality of $\mathcal{C}_r(G)$ implies that  
$I(\overline{\con_r(G)})$ has a linear resolution for all $r \ge 2$.
\end{proof}

\begin{theorem}\label{cycle-chordal}
If $G=C_n$ is a cycle on $n$ vertices, then the ideal $I(\overline{\con_r(G)})$ has a linear resolution for all $r\ge 2$.
\end{theorem}

\begin{proof}
Let $V(C_n)=\{1,\dots,n\}$ and 
$E(C_n)=\{\{1,2\},\{2,3\},\dots,\{n-1,n\},\{n,1\}\}.$

\medskip
\noindent
\textit{Case:} $n=r+1$.
Then $E(\mathcal{C}_r(C_n))=\{V\}$, where $V=\{1,\dots,r+1\}$.  
Let $\sigma=V\setminus\{r+1\}$. Since
\[
N_{\mathcal{C}_r(C_n)}[\sigma]
=
\sigma\cup\{v:\sigma\cup\{v\}\in E(\mathcal{C}_r(C_n))\}
=
V,
\]
and $V\in E(\mathcal{C}_r(C_n))$, we have $\sigma\in\mathrm{Sim}(\mathcal{C}_r(C_n))$.  
Moreover,
$\mathcal{C}_r(C_n)\setminus\{\sigma\}=\emptyset,$
so $\mathcal{C}_r(C_n)$ is chordal.

\medskip
\noindent
\textit{Case:} $n>r+1$.
Since $N_G(1)=\{2,n\}$, define
\[
\{e_1',\dots,e_m'\}
=
\{\, e\in E(\con_r(G))\mid \{1,2,n\}\subset e\,\},
\qquad
e_i=e_i'\setminus\{1\}.
\]
Because $1\in N_{\con_r(G)}[e_1]$ and $n> r+1$, we have
$N_{\con_r(G)}[e_1]=e_1',$
hence $e_1$ is simplicial and maximal in $\con_r(G)$.  
For $i>1$, let
$C=\con_r(G)\setminus\{e_1,\dots,e_{i-1}\}.$
As $e_j\neq e_i$ for all $j<i$, we have 
$e_j\not\subseteq e_i\cup\{1\},$
so $e_i\cup\{1\}=e_i'\in E(C)$, and thus
$N_C[e_i]=e_i'.$
Hence $e_i$ is simplicial and maximal in $C$.
Now define
\[
e_{m+1}=\{1,2,\dots,r\},\qquad
e_{m+2}=\{1,n,n-1,\dots,n-r+2\}.
\]
Then
\[
e_{m+1}\in \mathrm{Sim}\big(\con_r(G)\setminus\{e_1,\dots,e_m\}\big),
\qquad
e_{m+2}\in \mathrm{Sim}\big(\con_r(G)\setminus\{e_1,\dots,e_{m+1}\}\big).
\]

After deleting $e_1,\dots,e_{m+2}$ we obtain
$\con_r(G)\setminus\{e_1,\dots,e_{m+2}\}\cong \con_r(P_{n-1}),$
and by Theorem~\ref{block-p4}, $\mathcal{C}_r(P_{n-1})$ is chordal.  
Thus $\mathcal{C}_r(C_n)$ is chordal.

Finally, \cite[Theorem~3.3]{BYZ17} implies that chordality of $\mathcal{C}_r(C_n)$ yields a linear resolution of $I(\overline{\con_r(G)})$ for all $r\ge 2$.
\end{proof}

The next lemma provides a useful relationship between the complement of $\con_r(G)$ and the construction $\con_r(\overline{G})$.  
This inclusion will play an important role in the arguments that follow.

\begin{lemma}\label{inclusion}
Let $G$ be a graph and $r\ge1$. Then
$E\!\left(\overline{\con_r(G)}\right)\subseteq E\!\left(\con_r(\overline{G})\right),$
and for $r=1,2$,
$E\!\left(\overline{\con_r(G)}\right)=E\!\left(\con_r(\overline{G})\right).$
\end{lemma}

\begin{proof}
Let $\{a_1,\ldots,a_{r+1}\}\in E(\overline{\con_r(G)})$.  
Then it is not an edge of $\con_r(G)$, so $G[\{a_1,\ldots,a_{r+1}\}]$ is disconnected, whence 
$\overline{G}[\{a_1,\ldots,a_{r+1}\}]$ is connected.  
Thus $\{a_1,\ldots,a_{r+1}\}\in E(\con_r(\overline{G}))$, proving the inclusion.
For $r=1$ the equality is immediate from the definition of the complement of a graph.  
For $r=2$, we already have the inclusion above.  
Conversely, if $\{a,b,c\}\in E(\con_2(\overline{G}))$, then $\overline{G}[\{a,b,c\}]$ is connected, hence isomorphic to $P_3$ or $K_3$.  
In either case $G[\{a,b,c\}]$ is disconnected, so $\{a,b,c\}\in E(\overline{\con_2(G)})$.  
Therefore $E(\overline{\con_2(G)})=E(\con_2(\overline{G}))$.
\end{proof}

\begin{remark}\label{rem:reg-cycle}
It is worth noting that 
$\reg\!\bigl(I_1(\overline{C_n})\bigr)
   \;=\;
   \reg\!\bigl(I_2(\overline{C_n})\bigr)
   \;=\;3
   ~\text{for all } n\ge4.$
For $I_2(\overline{C_n})$, this follows immediately from Theorem~\ref{cycle-chordal} together with Lemma~\ref{inclusion}.  
The equality for $I_1(\overline{C_n})$ is established in~\cite[Theorem~2.6]{MYZ12}.
\end{remark}

The following lemma establishes a well-known connection between graph joins and disconnected complements.  
It provides a structural characterization that will be used in subsequent analysis.

\begin{lemma}\label{join-dis}
Let $G$ be a connected graph. Then 
\[
G = H_1 \vee \cdots \vee H_n \ (n\ge 2)
\quad\Longleftrightarrow\quad
\overline{G} \ \text{is disconnected}.
\]
\end{lemma}

\begin{proof}
Suppose 
$G = H_1 \vee \cdots \vee H_n,~ n\ge 2.$
Then
$\overline{G}
=
\overline{H_1 \vee \cdots \vee H_n}
=
\overline{H_1}\ \sqcup\ \cdots\ \sqcup\ \overline{H_n},$
since the complement of a join is the disjoint union of the complements.  
Hence $\overline{G}$ is disconnected.
Conversely, assume
$\overline{G}=G_1\sqcup\cdots\sqcup G_m,~m\ge 2.$
Taking complements,
$G=\overline{G_1}\ \vee\ \cdots\ \vee\ \overline{G_m},$
so $G$ is (isomorphic to) a join of graphs.
\end{proof}

We now provide a more precise description of the edges of $\overline{\con_r(G)}$ in terms of joins of induced subgraphs of the complement.  
This characterization will serve as a key tool in the structural and algebraic analysis developed in the subsequent sections.

\begin{lemma}\label{chara}
Let $G$ be a graph and let $r \ge 1$. Then

$E\bigl(\overline{\con_r(G)}\bigr)$
$=$
$\left\{
\{a_1,\dots,a_{r+1}\} \in E\bigl(\con_r(\overline{G})\bigr)\ \middle|\ 
\overline{G}[\{a_1,\dots,a_{r+1}\}] = B_1 \vee \cdots \vee B_m 
\text{ for some } m \ge 2
\right\},$

where each $B_i$ is an induced subgraph of $\overline{G}$.
\end{lemma}

\begin{proof}
Let $\{a_1,\ldots,a_{r+1}\}\in E(\overline{\con_r(G)})$.  
By Lemma~\ref{inclusion}, it also lies in $E(\con_r(\overline{G}))$.  
If $\overline{G}[\{a_1,\ldots,a_{r+1}\}]$ did not decompose as a join 
$B_1\vee\cdots\vee B_m$ with $m\ge2$, then by Lemma~\ref{join-dis} the induced subgraph  
$G[\{a_1,\ldots,a_{r+1}\}]$ would be connected, implying  
$\{a_1,\ldots,a_{r+1}\}\in E(\con_r(G))$, a contradiction.  
Hence such a decomposition exists.
Conversely, if $\{b_1,\ldots,b_{r+1}\}\in E(\con_r(\overline{G}))$ and  
$\overline{G}[\{b_1,\ldots,b_{r+1}\}]=B_1\vee\cdots\vee B_m$ for some $m\ge2$ and some graphs $B_i$,  
then Lemma~\ref{join-dis} implies that $G[\{b_1,\ldots,b_{r+1}\}]$ is disconnected.  
Thus $\{b_1,\ldots,b_{r+1}\}\in E(\overline{\con_r(G)})$.\end{proof}
Before introducing the next concept, we recall that edges of $\overline{\con_r(G)}$ containing simplicial vertices of $\overline{G}$ frequently arise from join decompositions of the corresponding induced subgraphs.  
We now formalize a minimal form of such decompositions.

\begin{definition}
Let $G$ be a graph, $r\ge1$, and let $e\in E(\overline{\con_r(G)})$ be an edge containing a simplicial vertex $x$ of $\overline{G}$.  
Assume that the induced subgraph $\overline{G}[e]$ admits a join decomposition
$\overline{G}[e]=B\vee D,$
where $D$ is complete and $x\in V(B)$.  
We say that $e$ is in \emph{reduced form} if no proper refinement exists; that is, there do not exist subgraphs $B'\subsetneq B$ and $D'\supsetneq D$ such that
\[
\overline{G}[e]=B'\vee D'
\quad\text{and}\quad
x\in V(B').
\]
\end{definition}

The next lemma guarantees the existence of a canonical join decomposition for edges of $\overline{\con_r(G)}$ that contain simplicial vertices.  
This result ensures that every such edge can be expressed in reduced form, providing a useful structural tool for the arguments developed later.

\begin{lemma}\label{simp-join}
Let $G$ be a graph, $r\ge 1$, and $e\in E(\overline{\con_r(G)})$ contain a simplicial vertex $x$ of $\overline{G}$. Then
\[
\overline{G}[e]=B\vee D,
\]
where $D$ is complete, $x\in V(B)$, and $e$ is in reduced form.
\end{lemma}

\begin{proof}
By Lemma~\ref{chara}, since $e \in E\!\big(\overline{\con_r(G)}\big)$, the induced subgraph 
$\overline{G}[e]$ admits a join decomposition
$\overline{G}[e]=B_1 \vee \cdots \vee B_m,~m\ge2.$
Assume without loss of generality that $x \in V(B_1)$, and set 
$D_1=\overline{G}[B_2 \vee \cdots \vee B_m].$
Because $x$ is simplicial in $\overline{G}$, its neighborhood forms a clique; hence every vertex of  
$B_2 \vee \cdots \vee B_m$ is adjacent to every other, so $D_1$ is complete.  
Thus
\[
\overline{G}[e]=B_1 \vee D_1,
\qquad x\in V(B_1),\ D_1 \text{ complete}.
\]

If this decomposition is already reduced, we take $B=B_1$ and $D=D_1$.  
Otherwise, we shrink $B_1$ to a proper induced subgraph $B\subsetneq B_1$ containing $x$ and enlarge $D_1$ to a corresponding induced supergraph $D\supsetneq D_1$, yielding another decomposition
$\overline{G}[e]=B\vee D$
with $x\in V(B)$ and $D$ complete.  
Hence $e$ is in reduced form.
\end{proof}

From Lemma~\ref{simp-join}, we derive additional properties of the decomposition $\overline{G}[e] = B \vee D$ when $e$ contains a simplicial vertex.  
These are summarized in the following observation.
\begin{obs}\label{obs-lemma}
    Let $G$ be a graph, $r \geq 1$, and $e \in E\big(\overline{\con_r(G)}\big)$ an edge containing a simplicial vertex $x$ of $\overline{G}$. By Lemma~\ref{simp-join}, we have a decomposition $\overline{G}[e] = B \vee D$, where
    $D$ is a complete graph, $x \in V(B)$, and $e$ is in reduced form.
     Then the following hold:
    \begin{enumerate}
        \item For any $z \in e$ with $\dd_{\overline{G}}(x,z)= 2$, we must have $z \in B$. 
        \begin{itemize}
            \item[Proof:] If $z \notin B$, then $z \in D$ and thus $\dd_{\overline{G}}(x,z) = 1$ (since $x$ is simplicial and adjacent to all vertices in $D$), contradicting the assumption.
        \end{itemize}
        
        \item For any $y \in N_{\overline{G}}(x) \cap B$, there exists $z \in B$ with $z \notin N_{\overline{G}}(y)$.
        \begin{itemize}
            \item[Proof:] If every $z \in B$ were adjacent to $y$, then we could move $y$ from $B$ to $D$, contradicting the reduced form condition.
        \end{itemize}
    \end{enumerate}
\end{obs}
The next lemma provides a decomposition principle for clutters containing a large complete join.  
It shows that, under mild hypotheses, one can recursively delete simplicial maximal subedges in a controlled manner until all edges containing a fixed subset $B$ are eliminated.

\begin{lemma}\label{clutter-decomposition}
Let $\mathcal{C}$ be a clutter satisfying:
\begin{enumerate}
    \item $E\!\bigl(\con_r(K_n \vee B)\bigr) \subseteq E(\mathcal{C})$, 
    where $B$ is a graph with $|B| = t \le r$, and 
    $\con_r(K_{n+t}) \cong \con_r(K_n \vee B) \neq \emptyset$;
    \item for every $r$-subset $e$ with $B \subseteq e$, the closed neighborhood satisfies 
    $N_{\mathcal{C}}[e] \subseteq K_n \vee B.$
\end{enumerate}
Then there exist $r$-subsets 
$e_1,\dots,e_\ell \subseteq V(K_n \vee B)$ such that:
\begin{itemize}
    \item $B \subseteq e_i$ for all $i$;
    \item $e_i \in \Sim\!\bigl(\mathcal{C} \setminus \{e_1,\dots,e_{i-1}\}\bigr)$ for $i=1,\ldots,l$;
    \item the clutter $\mathcal{C} \setminus \{e_1,\dots,e_\ell\}$ contains no edge that includes $B$.
\end{itemize}
\end{lemma}

\begin{proof}
Suppose $|B|=r$. 
Set $e_1=B$. Now $N_\mathcal{C}[B]= K_n\vee B$, hence $B$ is simplicial.  
Since $B$ is the only $r$-subsets containing $B$, removing $e_1$ deletes all such edges.

Suppose $|B|<r$.
Let $V(K_n)=\{x_1<\cdots<x_n\}$. Define
\[
\mathcal{E}
=\{\,e\subseteq V(K_n\vee B): B\cup\{x_1\}\subseteq e,\ |e|=r,\ x_n\notin e\,\},
\]
ordered lexicographically as $\mathcal{E}=\{e_1,\dots,e_m\}$.  
Let $\mathcal{C}_i:=\mathcal{C}\setminus\{e_1,\dots,e_{i-1}\}$.

For each $i$:
$x_n\in N_{\mathcal{C}_i}[e_i].$
For any $(r+1)$-subset $S\in N_{\mathcal{C}_i}[e_i]$, condition (2) implies $S\subseteq K_n\vee B$.  
For $j<i$, lex order gives $\exists\, a_j\in e_j\cap V(K_n)$ with $a_j\notin e_i\cap V(K_n)$ and $a_j<\max(e_i\cap V(K_n))$, hence
$e_j\not\subseteq S \quad(j<i),$
so $S\in E(\mathcal{C}_i)$ and $e_i$ is simplicial maximal in $\mathcal{C}_i$.

Let $\mathcal{C}'=\mathcal{C}\setminus\{e_1,\dots,e_m\}$.  
Then $\mathcal{C}'$ contains no edge including $B\cup\{x_1\}$.  
If $\mathcal{C}'$ contains no edge including $B$, we stop.
Otherwise, construct additional $r$-subsets $e_{m+1},\dots,e_s$ with
$B\cup\{x_2\}\subseteq e_j,$
and remove them to obtain $\mathcal{C}''=\mathcal{C}'\setminus\{e_{m+1},\dots,e_s\}$, which contains no edge including $B\cup\{x_2\}$.
Iterating this procedure for $x_1,\dots,x_n$, at stage $k$ removing all $r$-subsets including $B\cup\{x_k\}$ but excluding $\{x_1,\dots,x_{k-1}\}$, we eventually reach a clutter with no edge including $B$. Each $e_i$ added is simplicial in the current clutter by the same argument as above.

Thus the resulting sequence $e_1,\ldots,e_t$ satisfies
\[
B \subseteq e_i,\qquad
e_i \in \Sim\!\bigl(\mathcal{C}\setminus\{e_1,\ldots,e_{i-1}\}\bigr)
\ \text{for } i=1,\ldots,t,
\]
and
$\mathcal{C}\setminus\{e_1,\ldots,e_t\} 
\ \text{contains no edge that includes } B.$
\end{proof}

\section{Co–Chordal Clutters}\label{main}
In this section, we examine three graph classes that play a central role in our study: co-chordal-cactus graphs, $(2K_2, C_4)$-free graphs, co-grid graphs.
We show that, for each of these families, the associated clutter $\mathcal{C}_r(G)$ is co-chordal for all \(r \ge 2\), highlighting a common structural pattern underlying these distinct graph classes.

\subsection{Co-Chordal-Cactus Graphs} 
In this subsection we study the class of co-chordal-cactus graphs.  
We begin by recalling the notion of a cactus graph, which provides the structural foundation for this family.
A graph $T$ is called a \emph{cactus graph} if every edge of $T$ lies in at most one cycle; equivalently, any two cycles in $T$ share at most one vertex.  
A cactus graph need not be connected, and each connected component is itself a cactus graph.  
Trees constitute the special case in which $T$ contains no cycles; thus cactus graphs generalize trees while permitting cycles with minimal overlap.

We now introduce one of the principal objects of this paper.

\begin{definition}\label{def:co-chordal-cactus}
A graph $G$ is a \emph{co-chordal-cactus graph} if its complement $\overline{G}$ admits a representation
$(T, \{H_e\}_{e\in E(T)})$
satisfying:
\begin{enumerate}
    \item $T$ is a cactus graph, referred to as the \emph{skeleton} of $\overline{G}$.
    \item For each $e\in E(T)$, the subgraph $H_e$ is an induced subgraph of $\overline{G}$ and is either a cycle of length $\ge3$ or a chordal graph.
    \item 
    $
        V(\overline{G})=\bigcup_{e\in E(T)} V(H_e),~
        E(\overline{G})=\bigcup_{e\in E(T)} E(H_e).$
    \item The family $\{H_e\}_{e\in E(T)}$ satisfies the intersection property:  
    for any distinct $e,f\in E(T)$,  
    \[
        |V(H_e)\cap V(H_f)|\le1,
    \]
    and the intersection is nonempty only when $e$ and $f$ share a common endpoint in $T$.
\end{enumerate}
\end{definition}
To illustrate Definition~\ref{def:co-chordal-cactus}, consider a cactus graph 
$T$ serving as the skeleton of $\overline{G}$. In Figure \ref{fig:cactus}, the skeleton $T$ 
is depicted with its edges labelled $e_1,\ldots,e_9$. 
To each edge $e_i$, we associate an induced subgraph $H_{e_i}$, where 
$H_{e_i}$ is either a chordal graph or a cycle of length at least $3$. 
These subgraphs are attached in accordance with the intersection condition in 
Definition~\ref{def:co-chordal-cactus}: for distinct edges $e_i$ and $e_j$, 
the subgraphs $H_{e_i}$ and $H_{e_j}$ intersect in at most one vertex, and 
such an intersection occurs only when $e_i$ and $e_j$ share a common endpoint 
in $T$.

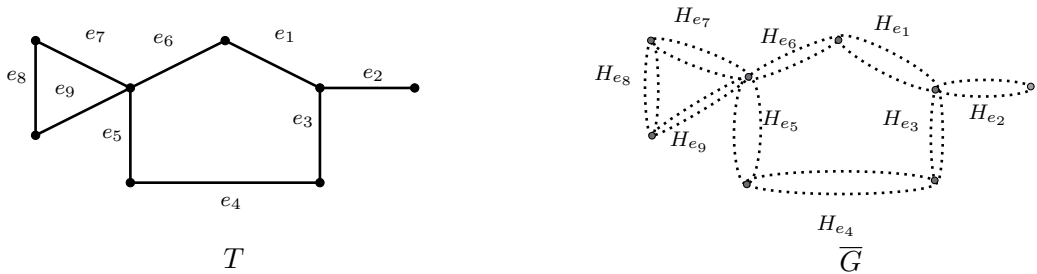
\begin{figure}[!h]
\centering
\definecolor{aqaqaq}{rgb}{0.6274509803921569,0.6274509803921569,0.6274509803921569}
\definecolor{wewdxt}{rgb}{0.43137254901960786,0.42745098039215684,0.45098039215686275}
\begin{tikzpicture}[scale=0.63]
\draw [line width=1pt] (5,12)-- (3,11);
\draw [line width=1pt] (3,11)-- (3,9);
\draw [line width=1pt] (3,9)-- (7,9);
\draw [line width=1pt] (7,11)-- (7,9);
\draw [line width=1pt] (7,11)-- (5,12);
\draw [line width=1pt] (7,11)-- (9,11);
\draw [line width=1pt] (3,11)-- (1,12);
\draw [line width=1pt] (1,12)-- (1,10);
\draw [line width=1pt] (3,11)-- (1,10);
\draw [line width=1pt] (9.06,15.16)-- (9.06,15.16);
\draw [rotate around={-27.4567422845999:(18.9800000000001,11.53)},line width=1pt,dotted] (18.9800000000001,11.53) ellipse (1.1727844743844804cm and 0.23264441396799876cm);
\draw [rotate around={0:(18,9)},line width=1pt,dotted] (18,9) ellipse (2.014349939504152cm and 0.24001183051099373cm);
\draw [rotate around={88.95837332399012:(16.02,10.1)},line width=1pt,dotted] (16.02,10.1) ellipse (1.13580730465548cm and 0.2822379019695034cm);
\draw [rotate around={24.128402930267875:(16.999999999999822,11.63)},line width=1pt,dotted] (16.999999999999822,11.63) ellipse (1.0604335430288985cm and 0.13423598320095464cm);
\draw [rotate around={89.81764678080121:(20.017515513726707,9.99462697224972)},line width=1pt,dotted] (20.017515513726707,9.99462697224972) ellipse (1.0025862406392378cm and 0.12604100570737045cm);
\draw [rotate around={0.3110653738672393:(21.01034054397408,10.994626972248463)},line width=1pt,dotted] (21.01034054397408,10.994626972248463) ellipse (1.005002286684463cm and 0.17485676398759442cm);
\draw [rotate around={-20.501073599331022:(15.022426931006876,11.617708398272732)},line width=1pt,dotted] (15.022426931006876,11.617708398272732) ellipse (1.107753399808703cm and 0.18871644546490898cm);
\draw [rotate around={90:(14,11)},line width=1pt,dotted] (14,11) ellipse (1.0098485524937781cm and 0.1406915028542663cm);
\draw [rotate around={31.13864923891614:(15.022426931006814,10.617708398272706)},line width=1pt,dotted] (15.022426931006814,10.617708398272706) ellipse (1.1987606499678831cm and 0.10053159384507897cm);
\draw (4.75,7.84) node[anchor=north west] {$T$};
\draw (17.75,7.88) node[anchor=north west] {$\overline{G}$};
\begin{scriptsize}
\draw [fill=black] (5,12) circle (2.5pt);
\draw [fill=black] (3,11) circle (2.5pt);
\draw[color=black] (3.71,12.01) node {$e_6$};
\draw [fill=black] (3,9) circle (2.5pt);
\draw[color=black] (2.63,10.01) node {$e_5$};
\draw [fill=black] (7,9) circle (2.5pt);
\draw[color=black] (5.13,8.57) node {$e_4$};
\draw [fill=black] (7,11) circle (2.5pt);
\draw[color=black] (6.65,10.31) node {$e_3$};
\draw[color=black] (6.27,12.09) node {$e_1$};
\draw [fill=black] (9,11) circle (2.5pt);
\draw[color=black] (8.13,11.27) node {$e_2$};
\draw [fill=black] (1,12) circle (2.5pt);
\draw[color=black] (2.27,12.09) node {$e_7$};
\draw [fill=black] (1,10) circle (2.5pt);
\draw[color=black] (0.61,11.23) node {$e_8$};
\draw[color=black] (1.63,10.93) node {$e_9$};
\draw[color=black] (19.1,12.31) node {$H_{e_1}$};
\draw[color=black] (17.9,8.09) node {$H_{e_4}$};
\draw[color=black] (16.74,10.31) node {$H_{e_5}$};
\draw[color=black] (16.7,12.05) node {$H_{e_6}$};
\draw[color=black] (19.28,10.31) node {$H_{e_3}$};
\draw[color=black] (21.1,10.43) node {$H_{e_2}$};
\draw [fill=wewdxt] (16.044853862013746,11.235416796545477) circle (2pt);
\draw[color=black] (14.88,12.49) node {$H_{e_7}$};
\draw[color=black] (13.2,11.27) node {$H_{e_8}$};
\draw[color=black] (14.8,9.87) node {$H_{e_9}$};
\draw [fill=wewdxt] (17.942252460207463,12.009483759288702) circle (2pt);
\draw [fill=aqaqaq] (22.00020091670044,11.030063456265212) circle (2pt);
\draw [fill=wewdxt] (19.972869073563256,9.048457017734071) circle (2pt);
\draw [fill=wewdxt] (16.007544281408595,8.964709190813096) circle (2pt);
\draw [fill=wewdxt] (14.008684703345164,9.992077263527607) circle (2pt);
\draw [fill=wewdxt] (13.984063300367364,12.003348944709932) circle (2pt);
\draw [fill=wewdxt] (19.987485538768542,10.962104741581218) circle (2pt);
\end{scriptsize}
\end{tikzpicture}
\caption{A co-chordal-cactus graph.}
\label{fig:cactus}
\end{figure}

In Figure \ref{fig:cactus}, each edge $e_i$ of the skeleton is shown together with a dotted 
region representing the corresponding subgraph $H_{e_i}$. The union of all 
these subgraphs $\{H_{e_i}\}_{i=1}^{9}$ yields the complement graph 
$\overline{G}$, thereby determining the associated co-chordal-cactus graph $G$.

\begin{remark}\label{rem:examples-cochordalcactus}
The class of co-chordal-cactus graphs contains several familiar graph families:

\begin{enumerate}
    \item Every \emph{co-chordal graph} is co-chordal-cactus: its complement $\overline{G}$ may be viewed as a single chordal piece $H_e$ with a skeleton $T$ consisting of one edge.

    \item Every \emph{co-cycle graph} is co-chordal-cactus, where $T$ is a single edge and $H_e$ is an induced cycle of length at least three.

    \item If $\overline{G}$ is a \emph{block graph} or a \emph{cactus graph}, then $G$ is co-chordal-cactus.  
    In this case, each block or cycle of $\overline{G}$ plays the role of one $H_e$, and their pattern of intersections determines the cactus skeleton $T$.

    \item More generally, any graph $G$ whose complement $\overline{G}$ is obtained by gluing chordal graphs and cycles along at most one shared vertex at a time belongs to the class of co-chordal-cactus graphs.
\end{enumerate}

Hence, co-chordal-cactus graphs unify co-chordal and co-cycle graphs and extend them to a broader class organized by a cactus skeleton.
\end{remark}

To describe the structural relationships among the constituent components arising in the decomposition of a co-chordal-cactus graph, we introduce the following terminology.
These notions will be used throughout the paper to analyze how vertices interact across different components and to capture the cyclic structure encoded in the skeleton of $\overline{G}$.

\begin{definition}\label{def:terminology}
Let $G$ be a co-chordal-cactus graph, and let 
$(T,\{H_e\}_{e\in E(T)})$ be the decomposition of its complement 
$\overline{G}$ as in Definition~\ref{def:co-chordal-cactus}.  
We use the following terminology.

\begin{enumerate}
    \item Each subgraph $H_e$ is called a \emph{constituent component} of $\overline{G}$.

    \item A vertex $v\in V(\overline{G})$ is a \emph{connector vertex} if it belongs to at least two constituent components, i.e.,
    $\bigl|\{\,e\in E(T): v\in V(H_e)\,\}\bigr|\ge2.$
\end{enumerate}
\end{definition}

To analyze structural transformations within co-chordal-cactus graphs, it is often useful to separate the connections of a vertex shared among multiple constituent components.  
This process, referred to as \emph{splitting}, replaces such a connector vertex by distinct copies, each confined to one component of the decomposition.  
The following definition formalizes this construction.

\begin{definition}\label{def:splitting-general}
Let $G$ be a co-chordal-cactus graph with decomposition  
$(T,\{H_e\}_{e\in E(T)})$ of its complement $\overline{G}$ as in
Definition~\ref{def:co-chordal-cactus}.  
Let $x$ be a connector vertex of $\overline{G}$, and let
$H_{i_1},\dots,H_{i_m}$ be the distinct constituent components containing $x$,  
where $m\ge2$ and
$x\in V(H_{i_1})\cap\cdots\cap V(H_{i_m}).$
For each $j\in\{1,\dots,m\}$, define
\[
N_{i_j}
    = \{\,v\in V(H_{i_j})\setminus\{x\} : \{x,v\}\in E(\overline{G})\,\},
\]
the set of neighbors of $x$ lying in $H_{i_j}$.
The \emph{splitting of $x$} in $\overline{G}$ is the graph $\overline{G}^{\,x}$
obtained by replacing $x$ with new vertices
$y_{i_1},\dots,y_{i_m}$, 
each inheriting precisely the neighbors of $x$ from its corresponding component. Formally,
\[
\begin{aligned}
V(\overline{G}^{\,x})
    &= (V(\overline{G})\setminus\{x\})\cup\{y_{i_1},\dots,y_{i_m}\},\\
E(\overline{G}^{\,x})
    &= E\bigl(\overline{G}[\,V(\overline{G})\setminus\{x\}\,]\bigr)
       \,\cup\,\bigcup_{j=1}^m \{\{y_{i_j},v\}: v\in N_{i_j}\}.
\end{aligned}
\]
All other adjacencies of $\overline{G}$ remain unchanged.
\end{definition}

The following lemma establishes a structural property of co-chordal-cactus graphs, describing how the splitting of a connector vertex in the complement interacts with the clutter $\con_r(G)$ construction.  
This decomposition plays a central role in analyzing the recursive structure and chordality of the associated clutter.

\begin{lemma}\label{main-lemma}
Let $G$ be a co-chordal-cactus graph with decomposition 
$(T,\{H_e\}_{e\in E(T)})$ of $\overline{G}$, and let $x$ be a connector vertex of $\overline{G}$ which is a cut vertex.  
Then there exist edges $e_1,\dots,e_m \in \overline{\mathcal{C}_r(G)}$ such that:
\begin{enumerate}
    \item 
    $e_i \in \Sim\!\left(\overline{\mathcal{C}_r(G)} \setminus \{e_1,\dots,e_{i-1}\}\right)$ 
    for all $i=1,\dots,m$;
    
    \item 
    $\overline{\mathcal{C}_r(G)}\setminus\{e_1,\dots,e_m\}
        \;\cong\;
        \overline{\mathcal{C}_r\!\left(\,\overline{\overline{G}_S^{\,x}}\,\right)}$.
\end{enumerate}
\end{lemma}

\begin{proof}
 Since $x$ is a cut vertex of $\overline{G}$, we have 
$\overline{G}\setminus\{x\}
=\bigsqcup_{i=1}^t \bigl(H_{e_i}\setminus\{x\}\bigr)$ for some constituent components $H_{e_1},\ldots,H_{e_t}$.
Let $e'\in E\bigl(\overline{\mathcal{C}_r(G)}\bigr)$ satisfy
\[
x\in e',\qquad 
e'\cap V(H_e)\ne\emptyset,\qquad
e'\cap V(H_f)\ne\emptyset,\qquad
e\neq f,   \ e,f\in\{e_1,\dots,e_t\}.
\]
Write $e=e'\setminus\{x\}$, and let $e_1',\dots,e_m'$ be all such edges; set  
$e_i=e_i'\setminus\{x\}$.  Define
\[
\mathcal{H}_0:=\overline{\mathcal{C}_r(G)},\qquad
\mathcal{H}_{i-1}:=\mathcal{H}_0\setminus\{e_1,\dots,e_{i-1}\}.
\]

\textsc{Claim.} $e_i\in\Sim(\mathcal{H}_{i-1})$ for all $i$.

\textit{Proof.}
The case $i=1$ is trivial.   Now
if $e_i\notin\Sim(\mathcal{H}_{i-1})$, then $\exists j<i$ with $e_j\subseteq e_i'$, hence
\[
e_j=e_j'\setminus\{x\}=e_i'\setminus\{x\}=e_i,
\]
contradiction.  
Thus
\[
N_{\mathcal{H}_{i-1}}[e_i]=e_i'\quad\Rightarrow\quad e_i\in\Sim(\mathcal{H}_{i-1}).
\]

After deleting $e_1,\dots,e_m$,
$\mathcal{H}_m
=
\overline{\mathcal{C}_r(G)}\setminus\{e_1,\dots,e_m\},$
which reflects the separation of $\overline{G}$ at $x$, so
$
\mathcal{H}_m
\cong
\overline{\mathcal{C}_r(\overline{\overline{G}_S^{\,x}})}.$
\end{proof}

We first establish the co-chordality of $\mathcal{C}_r(G)$ when the skeleton of the complement $\overline{G}$ is a path graph.  
This serves as the foundational case for more general skeleton structures.

\begin{theorem}\label{thm:path-cochordal}
Let $G$ be a co-chordal-cactus graph with decomposition $(T,\{H_e\}_{e\in E(T)})$ of $\overline{G}$ as in Definition~\ref{def:co-chordal-cactus}. If $T\cong P_n$, then $\mathcal{C}_r(G)$ is a co-chordal clutter.
\end{theorem}

\begin{proof}
We proceed by induction on $n = |V(T)|$.  
For the base case $n = 2$, let $H_e$ denote the unique constituent component of $\overline{G}$.

\textbf{Case 1: $H_e$ chordal.}  \\
If $\overline{\mathcal{C}_r(G)}$ is edgeless, done; assume $E(\overline{\mathcal{C}_r(G)})\ne\emptyset$.  
Let $x$ be simplicial in $\overline{G}$ and set
\[
\mathcal{B}=\{e\in E(\overline{\mathcal{C}_r(G)}):x\in e\},\quad 
A_e=\{z\in e:\dd_{\overline{G}}(x,z)=2\},\quad A=\bigcup_{e\in\mathcal{B}}A_e.
\]
By Lemma~\ref{simp-join}, for all $e\in\mathcal{B}$,
$\overline{G}[e]=B\vee D,\quad x\in B.$
Let $\{e_1',\dots,e_p'\}\subseteq\mathcal{B}$ be those with $e_i'\cap A\ne\emptyset$.  
Again Lemma~\ref{simp-join} yields
\[
\overline{G}[e_i']=B_i\vee C_i,\quad x\in B_i,\ \text{reduced},\ |B_i|\le r.
\]
Choose $k$ with $|B_k|=\max_i |B_i|$.

\textsc{Claim 1.}  
If $e\in E(\overline{\con_r(G)})$ and $B_k\subseteq e$, then $\overline{G}[e]=B_k\vee D$ for some complete graph $D$.

\textit{Proof.}  
Lemma~\ref{simp-join} gives $\overline{G}[e]=B\vee H$, $x\in B$, $H$ complete and $e$ is in reduced form.  
For $y\in B_k$:

(a) If $\dd_{\overline{G}}(x,y)=2$, then $y\in A$ and by Observation~\ref{obs-lemma}(1), $y\in B$.

(b) If $\dd_{\overline{G}}(x,y)=1$, Observation~\ref{obs-lemma}(2) gives $z\in B_k\cap A$ with $y\notin N_{\overline{G}}(z)$. If $y\in H$, then $y$ adjacent to $z$, contradiction; hence $y\in B$.

Thus $B_k\subseteq B$, and by maximality $B=B_k$. Set $D:=H$. \hfill$\square$

Let
\[
\mathcal{D}=\{f\in E(\overline{\con_r(G)}):B_k\subseteq f\}=\{f_1,\dots,f_q\},\quad 
\overline{G}[f_i]=B_k\vee D_i,\ D_i\ \text{complete}.
\]
Define $G_1 := D_1 \vee \cdots \vee D_q$.  
Then $\overline{G}[G_1]$ is a complete subgraph of $\overline{G}$.
For any $r$-subset $e'\subseteq G_1\vee B_k$ with $B_k\subseteq e'$, if $z'\in N_{\overline{\con_r(G)}}[e']\setminus e'$, then by \textsc{Claim 1},
\[
\overline{G}[e'\cup\{z'\}]=B_k\vee D_j,\quad \ \text{hence} \ z'\in D_j\subseteq G_1.
\]
Lemma~\ref{clutter-decomposition} yields $r$-subsets $e_{1,1},\dots,e_{1,n_1}$ such that each $e_{1,i}\in\Sim(\overline{\con_r(G)}\setminus\{e_{1,1},\dots,e_{1,i-1}\})$, each contains $B_k$, and
$\mathcal{C}'=\overline{\con_r(G)}\setminus\{e_{1,1},\dots,e_{1,n_1}\}$
has no edges containing $B_k.$
\vskip1mm
\noindent\textsc{Claim 2.}
There exist $r$-subsets $e_1,\dots,e_m$ such that:
\begin{enumerate}
\item 
$e_i\in\Sim\!\left(\overline{\con_r(G)}\setminus\{e_1,\dots,e_{i-1}\}\right) \ \text{for \ $i=1,\dots,m$}$;
\item 
$E\!\left(\overline{\con_r(G)}\setminus\{e_1,\dots,e_m\}\right)
=\mathcal{E}_1\;\coprod\;\mathcal{E}_2,
$
where
\[
\mathcal{E}_1=\Big\{\{a_1,\dots,a_{r+1}\}\in E(\con_r(\overline{G}\setminus x)):\ 
(\overline{G}\setminus x)[\{a_1,\dots,a_{r+1}\}]=B_1\vee\cdots\vee B_k,\ k\ge2\Big\},
\]
\[
\mathcal{E}_2=E\!\big(\con_r(N_{\overline{G}}[x])\big).
\]
\end{enumerate}

\vskip1mm
\noindent\textit{Proof of Claim 2.}
If $\mathcal{C}'$ has no $f$ with $x\in f$ and $f\cap A\ne\emptyset$, the claim is immediate.  
Otherwise let $f_1',\dots,f_q'$ be all such edges.  
By Lemma~\ref{simp-join},
\[
\overline{G}[f_i']=M_i\vee F_i,\qquad x\in F_i,\ |F_i|\le r \ \text{and each $f_i'$ \ is in reduced form}\ .
\]
Choose $F_k$ with $|F_k|=\max_i |F_i|$, so $|F_k|\le |B_k|$.
All edges of $\mathcal{C}'$ containing $F_k$ are of the form $F_k\vee E_j$ with $E_j$ complete and $x\in F_k$.  
Let these be $F_k\vee E_1,\dots,F_k\vee E_\ell$ and define
\[
\overline{G}[G_2]=\overline{G}[E_1\vee\cdots\vee E_\ell].
\] Also observe that $\overline{G}[G_2]$ is complete.
For any $r$-subsets $e\subseteq G_2\vee F_k$ with $F_k\subseteq e$ and any 
$z\in N_{\mathcal{C'}}[e]\setminus e$, we have $z\in G_2$.

For any $(r+1)$-subset $S\subseteq G_2\vee F_k$, $S$ is connected in $\overline{G}$ and satisfies a join decomposition.  

\textit{Claim:} $e_{1,i}\nsubseteq S$ for all $1\le i\le n_1$.

\textit{Subcase 1:} $\exists z\in A\cap B_k$ with $z\notin F_k$.  
Then $z\in e_{1,i}$ (since $B_k\subseteq e_{1,i}$) and $d_{\overline{G}}(x,z)=2$.  
If $z\in G_2$, then $z$ adjacent to $x$, contradiction.  
Thus $z\notin G_2\vee F_k$, hence $z\notin S$ and $e_{1,i}\nsubseteq S$.

\textit{Subcase 2:} $A\cap B_k\subseteq F_k$.  
Assume $e_{1,i}\subseteq S$.  
Then $B_k\subseteq S$.  
Since $B_k\ne F_k$, choose $y\in B_k\setminus F_k$; then $y\in N_{\overline{G}}(x)\cap B_k$.  
The previous argument forces $y\in G_2$.  
Thus $y\in N_{\overline{G}}(z)$ for all $z\in F_k\cap A$, and since $A\cap B_k\subseteq F_k$, we obtain $y\in N_{\overline{G}}(z)$ for all $z\in A\cap B_k$, contradicting Observation~\ref{obs-lemma}(2).  
Hence $e_{1,i}\nsubseteq S$ for all $i$.

By Lemma~\ref{clutter-decomposition}, $\exists\, e_{2,1},\dots,e_{2,n_2}$ such that
$\mathcal{C}'\setminus\{e_{2,1},\dots,e_{2,n_2}\}$
contains no edge including $F_k$,
$e_{2,i}\in\Sim\!\left(\mathcal{C}'\setminus\{e_{2,1},\dots,e_{2,i-1}\}\right).$

If no edge $h$ remains with $x\in h$ and $h\cap A\neq\emptyset$, then \textsc{Claim~2} holds.  
Otherwise, repeat the same procedure. After finitely many steps we obtain $r$-subsets
$e_{1,1},\dots,e_{p,n_p}$
such that
\[
E\!\left(\overline{\con_r(G)}\setminus\{e_{1,1},\dots,e_{p,n_p}\}\right)
=\mathcal{E}_1 \,\sqcup\, \mathcal{E}_2,
\]
where  
\[
\mathcal{E}_1=\Bigl\{\{a_1,\dots,a_{r+1}\}\subseteq V(\overline{G}\setminus\{x\}) :
\overline{G}[\{a_1,\dots,a_{r+1}\}]=B_1\vee\cdots\vee B_k,\ k\ge2\Bigr\},
\]
and  
$\mathcal{E}_2 = E\bigl(\con_r(N_{\overline{G}}[x])\bigr).$
Thus \textsc{Claim~2} follows.

\smallskip
We initially assume the existence of an edge $e\in E(\overline{\con_r(G)})$ with  
$e\cap A\neq\emptyset$.  
If no such edge exists, \textsc{Claim~2} is immediate.  
Likewise, if no edge containing $x$ exists, then \textsc{Claim~2} holds as well, in which case $\mathcal{E}_2=\emptyset$.

\textbf{Case 2: $H_e$ a cycle.} If $\overline{G}\cong C_n$, then for $r=2$ we have 
$E\bigl(\overline{\mathcal{C}_2(G)}\bigr)=E\bigl(\mathcal{C}_2(C_n)\bigr)$; hence, by Theorem~\ref{cycle-chordal}, the claim follows.  
For $r\ge3$, if $H_e\neq C_4$, then $E\bigl(\overline{\mathcal{C}_r(G)}\bigr)=\emptyset$, and the conclusion is immediate.  
When $H_e=C_4$, note that $|E(\overline{\con_3(C_4)})|=1$ and $E(\overline{\con_r(C_4)})=\emptyset$ for all $r\ge4$, so this case also poses no obstruction.  
Thus in every situation $\mathcal{C}_r(G)$ is co-chordal for all $r\ge2$, and the base case $n=2$ is established.

Assume the statement holds for all $P_k$ with $1\le k\le n$.  
Let 
$\overline{G}=(P_{n+1},\{H_e\}_{e\in E(P_{n+1})})$
be the decomposition corresponding to the skeleton $P_{n+1}$.  
Since $P_{n+1}$ is connected and acyclic, $\overline{G}$ contains a connector vertex $x$, which is necessarily a cut vertex.
 By Lemma~\ref{main-lemma}, splitting $x$ produces a graph whose associated clutter $\con_r(G)$ is co-chordal, because the splitting replaces the single path skeleton by a disjoint union of shorter paths; by the induction hypothesis, each such component yields a co-chordal $\con_r(G)$.  
Therefore $\con_r(G)$ is co-chordal for $\overline{G}$ as well, completing the proof.
\end{proof}

Next, we consider the case where the skeleton of the complement $\overline{G}$ is a cycle.  

\begin{theorem}\label{thm:cycle-cochordal}
Let $G$ be a co-chordal-cactus graph with decomposition $(T,\{H_e\}_{e\in E(T)})$ of $\overline{G}$ as in Definition~\ref{def:co-chordal-cactus}.  
If $T\cong C_n$, then $\con_r(G)$ is a co-chordal clutter for all $r\ge2$.
\end{theorem}

\begin{proof}
Assume throughout the proof
that inner-cycle lengths are preserved when replacing an edge of the skeleton by a chordal graph or cycle.

\textbf{Case $n=3$.}  
Let $H_1,H_2,H_3$ be the components corresponding to $C_3$.  
If all $H_i$ are chordal then $\overline{G}$ is chordal, hence the claim follows from Theorem~\ref{thm:path-cochordal}.  
Assume $\overline{G}$ is not chordal. Hence one of $H_1,H_2,H_3$  is cycle. Without loss of generality, let $H_3$ is a cycle. Let $V(C_3)=\{x_1,x_2,x_3\}$.

\medskip
\emph{Case 1: $H_3$ a 4-cycle.}  
Let $V(H_3)=\{x_1,x_3,x_4,x_5\}$.  

$\bullet$ \emph{$r=2$.}  
Choose $e_{1}=\{x_{1},x_{4}\}$.  
Since
$N_{\overline{\con_2(G)}}[e_{1}]
=\{x_{1},x_{3},x_{4},x_{5}\},$
we have $e_{1}\in \Sim(\overline{\con_2(G)})$.  
Let $\mathcal{C}=\overline{\con_2(G)}\setminus e_{1}$.  
Next choose $e_{2}=\{x_{5},x_{4}\}$.  
Because
$N_{\mathcal{C}}[e_{2}]
=\{x_{5},x_{4},x_{3}\},$
we also have $e_{2}\in \Sim(\mathcal{C})$.
Deleting $e_{1}$ and $e_{2}$ yields a graph $G'$ such that
$\overline{\con_2(G)}\setminus\{e_{1},e_{2}\}
\;\cong\;
\overline{\con_2(\overline{G'})},$
where
\[
E(G') 
\cong 
E\bigl(\overline{G}\setminus\{x_{5}\}\bigr)
\;\cup\;
\{\{x_{1},y_{1}\}\},
\qquad
V(G')
=
V\bigl(\overline{G}\setminus\{x_{5}\}\bigr)
\;\cup\;
\{y_{1}\}.
\]

$\bullet$ \emph{$r=3$.}  
Take $e_{1}=\{x_{1},x_{5},x_{4}\}$.  
Since
$N_{\overline{\con_3(G)}}[e_{1}]
=\{x_{1},x_{3},x_{4},x_{5}\},$
we have $e_{1}\in \Sim(\overline{\con_3(G)})$.  
Deleting $e_{1}$ produces the same graph $G'$ as in the case $r=2$, and
\[
\overline{\con_3(G)}\setminus\{e_{1}\}
\;\cong\;
\overline{\con_3(\overline{G'})}.
\]

$\bullet$ \emph{$r\ge4$.}  
No edge of $\overline{\con_r(G)}$ can contain any of the vertices $x_{1},x_{4},x_{5}$, since every such edge would necessarily contain an induced $P_{4}$, contradicting the join–decomposition constraints.  
Hence
\[
\overline{\con_r(G)}
\;\cong\;
\overline{\con_r(\overline{G'})}.
\]

\medskip
\emph{Case 2: $H_{3}$ a cycle of length $\ell \ge 5$.}  
Let $V(H_{3})=\{x_{1},x_{3},x_{4},\dots,x_{p}\}$.

\smallskip
\noindent$\bullet$ \emph{$r=2$, subcase $\ell=5$.}  
Choose
\[
e_{1}=\{x_{4},x_{6}\},\qquad 
e_{2}=\{x_{5},x_{3}\},\qquad
e_{3}=\{x_{5},x_{1}\}.
\]
Then
\[
e_{i}\in 
\Sim\!\left(
\overline{\C_{2}(G)}\setminus\{e_{1},\dots,e_{i-1}\}
\right)
\qquad (i=1,2,3),
\]
and
\[
E\!\left(\overline{\C_{2}(G)}\setminus\{e_{1},e_{2},e_{3}\}\right)
=
E\!\left(\overline{\con_{2}(G\setminus\{x_{5}\})}\right).
\]

\smallskip
\noindent$\bullet$ \emph{$r=2$, subcase $\ell\ge6$.}  
Choose
\[
e_{1}=\{x_{4},x_{6}\},\qquad 
e_{2}=\{x_{5},x_{3}\},\qquad
e_{3}=\{x_{5},x_{7}\}.
\]
Then
\[
e_{i}\in 
\Sim\!\left(
\overline{\C_{2}(G)}\setminus\{e_{1},\dots,e_{i-1}\}
\right)
\qquad (i=1,2,3),
\]
and again
\[
E\!\left(\overline{\C_{2}(G)}\setminus\{e_{1},e_{2},e_{3}\}\right)
=
E\!\left(\overline{\con_{2}(G\setminus\{x_{5}\})}\right).
\]

$\bullet$ \emph{$r \ge 3$.}  
No edge of $\overline{\con_r(G)}$ can contain the vertex $x_5$, since any such edge would necessarily include an induced $P_4$, violating the join–decomposition constraints.  
Consequently,
\[
E\bigl(\overline{\con_r(G)}\bigr)
=
E\bigl(\overline{\con_r(G\setminus\{x_5\})}\bigr),
\]
and the argument reduces to a smaller graph.
If either $H_1$ or $H_2$ is a cycle, then—as in the previous cases—the cycle can be broken by deleting suitable $r$-subsets.  
Thus there exist $r$-subsets $e_1,\dots,e_m$ such that  
\[
e_i \in 
\Sim\!\left(\overline{\con_r(G)}\setminus\{e_1,\dots,e_{i-1}\}\right)
\qquad (i=1,\dots,m),
\]
and
\[
\overline{\con_r(G)}\setminus\{e_1,\dots,e_m\}
=
\overline{\con_r(\overline{G''})},
\]
where $G''$ is a chordal graph—in fact, a unicyclic graph whose unique cycle is a triangle.

\medskip
\textbf{Case $n=4$.}  
Let $V(C_4)=\{x_1,x_2,x_3,x_4\}$, and let 
$H_{12}, H_{23}, H_{34}, H_{41}$ denote the corresponding constituent components of $\overline{G}$.  
Define
\[
\{e_1',\dots,e_m'\}
=
\Bigl\{
\, e' \in E\bigl(\overline{\con_r(G)}\bigr)
\ \Big|\
x_1 \in e',\
e' \cap V(H_{12} \setminus \{x_1\}) \neq \emptyset,\
e' \cap V(H_{41} \setminus \{x_1\}) \neq \emptyset
\Bigr\}.
\]
For each $i$, set  
$e_i = e_i' \setminus \{x_1\}.$

\medskip
\noindent\textit{Case $r\ge3$.}
Let $e_i\subseteq V(\overline{G})$ be as above.  
If $x_2\notin e_i$, then $x_3\notin N_{\overline{\con_r(G)}}[e_i]$; otherwise $e_i\cup\{x_3\}$ would contain an induced $P_4$, which is impossible under the join–decomposition constraints.  
Assume now that $x_2,x_4\in e_i$.  
Since $e_i$ intersects both $H_{12}\setminus\{x_1\}$ and $H_{41}\setminus\{x_1\}$, there exists
\[
z\in V(H_{12}\setminus\{x_1\}) 
\quad\text{or}\quad 
z\in V(H_{41}\setminus\{x_1\});
\]
without loss of generality, take $z\in H_{12}\setminus\{x_1\}$.  
If $x_3\in N_{\overline{\con_r(G)}}[e_i]$, then the vertices 
$z,\ x_1,\ x_4,\ x_3$
form a shortest path $P_4 = z\,x_1\,x_4\,x_3$ in $\overline{G}$.  
Since $P_4$ cannot arise as a join of two graphs, and no vertex $w\in\overline{G}$ is adjacent to all of 
$z,x_1,x_4,x_3$, this contradicts the fact that
\[
\overline{G}[\,e_i\cup\{x_3\}\,] = B_1\vee\cdots\vee B_k,\qquad k\ge2.
\]
Thus $x_3 \notin N_{\overline{\con_r(G)}}[e_i]$.  
A symmetric argument shows that if $z\in N_{\overline{\con_r(G)}}[e_i]$, then necessarily $z=x_1$.

Let  
$\mathcal{C}
:= \overline{\con_r(G)} \setminus \{e_1,\dots,e_{i-1}\}.$
Then $N_{\mathcal{C}}[e_i]=e_i'$, and therefore
\[
E\bigl(\overline{\con_r(G)}\setminus\{e_1,\dots,e_m\}\bigr)
=
E\bigl(\overline{\con_r(\,\overline{\overline{G}^{\,x_1}}\,)}\bigr).
\]
By Theorem~\ref{thm:path-cochordal}, $\con_r(G)$ is co-chordal.
If instead $x_2\in e_i$ and $x_4\notin e_i$, or if $x_2,x_4\notin e_i$, the same reasoning applies verbatim and yields the same conclusion.

\noindent
\noindent\emph{Case $r=2$.}
If $x_2\notin e_i$, then necessarily $x_3\notin N_{\overline{\con_2(G)}}[e_i]$.  
If $x_2,x_4\in e_i$, then $e_i=\{x_2,x_4\}$ and
$N_{\overline{\con_2(G)}}[e_i]=\{x_1,x_2,x_3,x_4\}.$
In all remaining configurations,
$N_{\overline{\con_2(G)}}[e_i]=\{x_1\}.$
Thus every successive deletion is simplicial, and $\con_2(G)$ is co-chordal.

\medskip
\textbf{Case $n\ge5$.}
Let $H_1,\dots,H_n$ be the components of $\overline{G}$, and assume  
$V(H_1)\cap V(H_2)=\{x\}$.  
Let $e_1',\dots,e_m'$ be the edges of $\overline{\con_r(G)}$ satisfying
\[
x\in e_i', \qquad 
e_i'\cap V(H_1)\neq\emptyset, \qquad 
e_i'\cap V(H_2)\neq\emptyset.
\]
Set $e_i = e_i'\setminus\{x\}$. Then
\[
E\!\left(\overline{\con_r(G)}\setminus\{e_1,\dots,e_m\}\right)
=
E\!\left(\overline{\con_r\!\left(\overline{\,\overline{G}^{\,x}}\,\right)}\right),
\]
and each $e_i$ is simplicial:
\[
e_i \in 
\Sim\!\left(\overline{\con_r(G)}\setminus\{e_1,\dots,e_{i-1}\}\right)
\qquad (i=1,\dots,m).
\]

Deleting the connector vertex $x$ produces a smaller co-chordal-cactus graph whose skeleton is a path.  
By Theorem~\ref{thm:path-cochordal}, $\con_r(G)$ is therefore co-chordal.

Consequently, for all $n\ge3$ and $r\ge2$, the clutter $\con_r(G)$ is co-chordal.
\end{proof}

We now state one of the main results of this work, which establishes the co-chordality of the clutter associated with any co-chordal-cactus graph.

\begin{theorem}\label{thm:main-cochordal}
Let $G$ be a co-chordal-cactus graph. Then $\con_r(G)$ is a co-chordal clutter.
\end{theorem}

\begin{proof}
Let $(T,\{H_e\}_{e\in E(T)})$ be the decomposition of $\overline{G}$ as in Definition~\ref{def:co-chordal-cactus}.  
If $T$ has no cut vertex, then $T$ is either a cycle or a single edge, and the claim follows from Theorems~\ref{thm:cycle-cochordal} and~\ref{thm:path-cochordal}.  
If $T$ has a cut vertex, then $\overline{G}$ contains a connector vertex.  
By repeated application of Lemma~\ref{main-lemma}, together with Theorems~\ref{thm:cycle-cochordal} and~\ref{thm:path-cochordal}, each splitting at a connector cut vertex preserves co-chordality while strictly reducing the co-chordal-cactus structure.  
Iterating this reduction eventually reaches a terminal configuration covered by the base cases, and hence $\con_r(G)$ is co-chordal.
\end{proof}

As an immediate consequence of Theorem~\ref{thm:main-cochordal}, we are able to verify \cite[Conjecture~6.1]{DRSV24}, which yields the following corollary.

\begin{corollary}
Let $G$ be a co-chordal graph. Then, for every integer $r \ge 1$, the clutter $\con_r(G)$ is co-chordal.
\end{corollary}

\begin{remark}
An algebraic interpretation of \cite[Conjecture~6.1]{DRSV24} is also possible.  
If $G$ is co-chordal, then $I_r(G)=I_{\ind_r(G)}$ is vertex splittable for all $r\ge1$ by \cite[Theorem~3.12]{DRSV24}.  
A simplicial complex $\ind_r(G)$ is vertex decomposable iff its Alexander dual ideal $I_{\ind_r(G)^\vee}$ is vertex splittable \cite[Theorem~2.3]{MorKho16}.  
Assuming $\con_r(G)$ is a clutter and that the simplicial complex 
$\Delta(\overline{\con_r(G)})^\vee=\ind_r(G)^\vee$  
is vertex decomposable (where $\Delta(\cdot)$ denotes the clique complex of a clutter),  
\cite[Theorem~3.10]{Ni19} implies that $\overline{\con_r(G)}$ is chordal.
\end{remark}

\subsection{\texorpdfstring{$(2K_2, C_4)$-Free Graphs}{(2K2, C4)-Free Graphs}}

In this subsection, we establish the co-chordality property for another important class of graphs.  
Specifically, we show that $(2K_2, C_4)$-free graphs also give rise to co-chordal clutters.  
It is worth noting that when $r = 1$, a $(2K_2, C_4)$-free graph does not necessarily yield a co-chordal clutter $\con_1(G)$.  
However, for all $r \ge 2$, the associated clutter $\con_r(G)$ is co-chordal.

\begin{theorem}\label{2k2-c4-free}
Let $G$ be a $(2K_2,C_4)$-free graph. Then $\con_r(G)$ is a co-chordal clutter for all $r\ge2$.
\end{theorem}

\begin{proof}
By the structural description of $(2K_2,C_4)$-free graphs \cite[Theorem~1.1]{BHP93}, $V(G)=V_1\sqcup V_2\sqcup V_3$ satisfies:
(i) $V_1$ independent,  
(ii) $V_2$ a clique,  
(iii) $V_3=\emptyset$ or $|V_3|=5$, in which case $V_3$ induces a $C_5$,  
(iv) if $V_3\neq\emptyset$, then for all $v_1\in V_1$, $v_2\in V_2$, $v_3\in V_3$:  
$\{v_1,v_3\}\notin E(G)$ and $\{v_2,v_3\}\in E(G)$.

Passing to the complement, $\overline{G}$ is also $(2K_2,C_4)$-free and admits the same partition, with:
(i$'$) $\overline{V_2}$ independent,  
(ii$'$) $\overline{V_1}$ a clique,  
(iii$'$) $\overline{V_3}=\emptyset$ or $|\overline{V_3}|=5$, inducing again a $C_5$,  
(iv$'$) if $\overline{V_3}\neq\emptyset$, then for all $v_1\in\overline{V_1}$, $v_2\in\overline{V_2}$, $v_3\in\overline{V_3}$:  
$\{v_1,v_3\}\in E(\overline{G})$ and $\{v_2,v_3\}\notin E(\overline{G})$.

If $\overline{V_3}=\emptyset$, then $\overline{G}$ is co\-chordal by Theorem~\ref{thm:path-cochordal}, hence $\con_r(G)$ is co-chordal.
 Hence assume $\overline{V_3}\neq\emptyset$.  
Then (iv$'$) implies every $x\in\overline{V_2}$ is simplicial in $\overline{G}$. Using the method from Theorem~\ref{thm:path-cochordal}, there exists a sequence $e_1,\dots,e_t\in E(\overline{\C_r(G)})$ such that
\[
\begin{aligned}
E\!\left(\C_r(\overline{G})\setminus\{e_1,\dots,e_t\}\right)
=\;&
\Bigl\{\{a_1,\dots,a_{r+1}\}\in E(\C_r(\overline{G})):\,
\overline{G}\setminus\overline{V_2}[\{a_1,\dots,a_{r+1}\}]  \\
&=\, B_1\vee\dots\vee B_m,\ m\ge2 \Bigr\}.
\end{aligned}
\]
Let $\overline{V_3}=\{1,2,3,4,5\}$ in cyclic order. We now treat cases by $r$.

\medskip
\noindent\textit{Case 1: $r=2$.}  
Remove the following ten simplicial edges in order:
\[
\begin{aligned}
&e_{t+1}=\{1,3\},\ e_{t+2}=\{1,4\},\ e_{t+3}=\{2,4\},\ e_{t+4}=\{2,5\},\ e_{t+5}=\{3,5\},\\
&e_{t+6}=\{1,2\},\ e_{t+7}=\{2,3\},\ e_{t+8}=\{3,4\},\ e_{t+9}=\{4,5\},\ e_{t+10}=\{1,5\}.
\end{aligned}
\]
After these deletions, any remaining edge $e$ satisfies $|e\cap\overline{V_3}|\le1$.  
Since $\overline{G}[\{1\}\cup\overline{V_1}]$ is complete and all its edges lie in $\overline{\C_2(G)}$, Lemma~\ref{compl-chordal} yields edges $e_{t+11},\dots,e_s$ eliminating all edges containing $1$. Repeating for $2,3,4,5$ produces $e_{s+1},\dots,e_p$ so that
$E\!\left(\overline{\C_2(G)}\setminus\{e_1,\dots,e_p\}\right)=E(\C_2(\overline{V_1})).$
As $\overline{V_1}$ induces a clique, another application of Lemma~\ref{compl-chordal} gives $e_{p+1},\dots,e_q$ with
$E(\overline{\C_2(G)}\setminus\{e_1,\dots,e_q\})=\emptyset$.
Thus $\C_2(G)$ is co-chordal.

\medskip
\noindent\textit{Case 2: $r=3$.}
Let $V(\overline{V_1})=\{6,\dots,n\}$.  
Let $\mathcal{E}=\{e_1<\cdots<e_m\}$ be the collection of all $3$-subsets 
$e_i$ with $1\in e_i$ and $n\notin e_i$, ordered lexicographically.
 Suppose $\{a,b,c\}\in \mathcal{E}$.
If $z\in N_{\overline{\C_3(G)}}(\{a,b,c\})$ for some $a,b,c\in\{1,\dots,5\}$, then  
$z\notin\overline{V_3}$; otherwise the induced subgraph on $\{a,b,c,z\}$ would contain a $P_4$, which cannot arise from a join decomposition.  
Proceeding as in Lemma~\ref{compl-chordal}, there exists an index $s'$ such that  
$\overline{\C_3(G)}\setminus\{e_1,\dots,e_{s'}\}$ contains no edge involving the vertex $1$, and
\[
e_i\in \Sim\!\left(\overline{\C_3(G)}\setminus\{e_1,\dots,e_{i-1}\}\right)
\qquad (i=1,\dots,s').
\]
By repeating the same argument for the remaining vertices in $\{2,3,4,5\}$, we obtain additional simplicial edges  
$e_{s'+1},\dots,e_t$ such that  
$E\!\left(\overline{\C_3(G)}\setminus\{e_1,\dots,e_t\}\right)
=E(\overline{V_2}).$

Finally, Lemma~\ref{compl-chordal} provides the last deletions  
$e_{t+1},\dots,e_p$, which remove all edges of $\overline{\C_3(G)}$.  
Thus $\C_3(G)$ is co-chordal.

\medskip
\noindent\textit{Case 3: $r=4$.}  
Same argument as \textit{Case~2.}

\medskip
\noindent\textit{Case 4: $r\ge5$.}
In this range we have 
$E\bigl(\overline{\C_r(G)}\bigr)=E\bigl(\C_r(K_n)\bigr)$ for some $n$.
Since $\C_r(K_n)$ is co-chordal by Lemma~\ref{compl-chordal}, the desired conclusion follows immediately.

Thus for all $r\ge2$, $\C_r(G)$ is co-chordal.
\end{proof}

\subsection{Co-Grid Graphs:}
In this subsection, we study the class of \emph{co-grid graphs}, which arise as complements of grid graphs.  
Grid graphs form a fundamental family in graph theory, constructed as Cartesian products of path graphs.  
We first recall the definition of a grid graph and illustrate it with a simple example.
The \emph{grid graph} $G_{n,m} = P_n \square P_m$ is the Cartesian product of two path graphs $P_n$ and $P_m$. 
Its vertex set is 
$V(G_{n,m}) = \{\, x_{i,j} \mid 1 \le i \le n,\; 1 \le j \le m \,\},$
and two vertices $x_{i,j}$ and $x_{k,\ell}$ are adjacent if and only if
$|i - k| + |j - \ell| = 1.$ An example of the grid graph $G_{3,4}$ is illustrated in Figure~\ref{fig:gridgraph}.

\begin{figure}[!h]
\centering
\begin{tikzpicture}[scale=0.85]
\draw [line width=1pt] (8,11)-- (9,11);
\draw [line width=1pt] (9,11)-- (10,11);
\draw [line width=1pt] (10,11)-- (11,11);
\draw [line width=1pt] (8,11)-- (8,10);
\draw [line width=1pt] (8,10)-- (9,10);
\draw [line width=1pt] (9,11)-- (9,10);
\draw [line width=1pt] (10,11)-- (10,10);
\draw [line width=1pt] (9,10)-- (10,10);
\draw [line width=1pt] (11,11)-- (11,10);
\draw [line width=1pt] (10,10)-- (11,10);
\draw [line width=1pt] (8,10)-- (8,9);
\draw [line width=1pt] (9,10)-- (9,9);
\draw [line width=1pt] (10,10)-- (10,9);
\draw [line width=1pt] (11,10)-- (11,9);
\draw [line width=1pt] (8,9)-- (9,9);
\draw [line width=1pt] (9,9)-- (10,9);
\draw [line width=1pt] (10,9)-- (11,9);
\begin{scriptsize}
\draw [fill=black] (8,11) circle (2.5pt);
\draw[color=black] (8.077300275482095,11.381046831955883) node {$x_{1,1}$};
\draw [fill=black] (9,11) circle (2.5pt);
\draw[color=black] (9.069035812672178,11.403085399448997) node {$x_{1,2}$};
\draw [fill=black] (10,11) circle (2.5pt);
\draw[color=black] (10.038732782369147,11.403085399448997) node {$x_{1,3}$};
\draw [fill=black] (11,11) circle (2.5pt);
\draw[color=black] (11.019449035812674,11.414104683195553) node {$x_{1,4}$};
\draw [fill=black] (8,10) circle (2.5pt);
\draw[color=black] (7.713663911845732,10.334214876033023) node {$x_{2,1}$};
\draw [fill=black] (9,10) circle (2.5pt);
\draw[color=black] (8.6833608815427,10.367272727272692) node {$x_{2,2}$};
\draw [fill=black] (10,10) circle (2.5pt);
\draw[color=black] (9.653057851239671,10.367272727272692) node {$x_{2,3}$};
\draw [fill=black] (11,10) circle (2.5pt);
\draw[color=black] (10.688870523415979,10.378292011019248) node {$x_{2,4}$};
\draw [fill=black] (8,9) circle (2.5pt);
\draw[color=black] (7.647548209366393,9.32044077134983) node {$x_{3,1}$};
\draw [fill=black] (9,9) circle (2.5pt);
\draw[color=black] (8.65030303030303,9.364517906336056) node {$x_{3,2}$};
\draw [fill=black] (10,9) circle (2.5pt);
\draw[color=black] (9.653057851239671,9.3534986225895) node {$x_{3,3}$};
\draw [fill=black] (11,9) circle (2.5pt);
\draw[color=black] (10.65581267217631,9.386556473829168) node {$x_{3,4}$};
\end{scriptsize}
\end{tikzpicture}
\caption{The grid graph \(G_{3,4}\).}
\label{fig:gridgraph}
\end{figure}
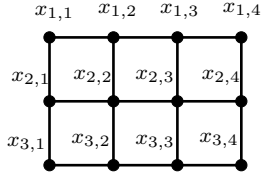

\begin{theorem}\label{thm:con-cochordal}
Let $G=\overline{G_{n,m}}$ with $n,m\ge2$. Then $\con_r(G)$ is a co-chordal clutter for all $r\ge2$.
\end{theorem}

\begin{proof}
Let $V(G)=\{x_{i,j}\mid 1\le i\le n,\ 1\le j\le m\}$.  
We first treat $n=2$, proving the claim for all $r\ge2$ by induction on $m$.
Suppose $m=2$.
Then $\overline{G}=C_4$.  
For $r=2$, apply Theorem~\ref{cycle-chordal} and Lemma~\ref{inclusion};  
for $r=3$, $|E(\overline{\con_3(G)})|=1$;  
for $r\ge4$, $\overline{\con_r(G)}=\emptyset$.  
Thus the result holds.

\medskip
\noindent\textit{Induction on $m$ for $n=2$.}  
Assume the theorem holds for $\overline{G_{2,k}}$ with $k\le m$, and set $G=\overline{G_{2,m+1}}$.

\smallskip
\noindent\emph{Case $r=2$.}  
Edges of $\overline{\con_2(G)}$ containing $x_{1,1}$ are
\[
\{x_{1,1},x_{1,2},x_{1,3}\},\ 
\{x_{1,1},x_{1,2},x_{2,1}\},\ 
\{x_{1,1},x_{1,2},x_{2,2}\},\ 
\{x_{1,1},x_{2,1},x_{2,2}\}.
\]
Choose $e_1=\{x_{1,1},x_{1,3}\}$, $e_2=\{x_{1,1},x_{2,2}\}$, $e_3=\{x_{1,1},x_{1,2}\}$, and 
put  
$\mathcal{C}=\overline{\con_2(G)}\setminus\{e_1,e_2,e_3\}$.  
Edges of $\mathcal{C}$ containing $x_{2,1}$ are  
$\{x_{2,1},x_{2,2},x_{2,3}\}$ and $\{x_{2,1},x_{2,2},x_{1,2}\}$.  
Choose $e_4=\{x_{2,1},x_{2,3}\}$ and $e_5=\{x_{2,1},x_{1,2}\}$. Observe that \[
e_i\in \Sim\!\left(\overline{\C_2(G)}\setminus\{e_1,\dots,e_{i-1}\}\right)
\qquad i=1,\dots,5.
\]
and 
$E(\overline{\con_2(G)}\setminus\{e_1,\dots,e_5\}) 
= E(\overline{\con_2(\overline{G_{2,m}})}),$
and induction yields co-chordality.

\smallskip
\medskip
\noindent\textit{Case $r=3$.}
The edges of $\overline{\con_3(G)}$ containing $x_{1,1}$ are
\[
\{x_{1,1},x_{1,2},x_{2,1},x_{2,2}\},\qquad
\{x_{1,1},x_{1,2},x_{1,3},x_{2,2}\}.
\]
Select
$e_1=\{x_{1,1},x_{1,3},x_{2,2}\}$,
$e_2=\{x_{1,1},x_{2,1},x_{2,2}\},$
and set \(\mathcal{C}=\overline{\con_3(G)}\setminus\{e_1,e_2\}\).  
Among the edges of \(\mathcal{C}\) containing \(x_{2,1}\) is
$\{x_{2,1},x_{2,2},x_{2,3},x_{1,2}\}.$
Choose 
$e_3=\{x_{2,1},x_{2,3},x_{1,2}\}.$
After deleting \(e_1,e_2,e_3\), we obtain
\[
E\!\left(\overline{\con_3(G)}\setminus\{e_1,e_2,e_3\}\right)
=
E\!\left(\overline{\con_3(\overline{G_{2,m}})}\right),
\] and \[
e_i\in \Sim\!\left(\overline{\C_3(G)}\setminus\{e_1,\dots,e_{i-1}\}\right)
\qquad i=1,2,3.
\]
so the induction hypothesis applies.

\smallskip
\noindent\emph{Case $r\ge4$.}  
Then $\overline{\con_r(G)}=\emptyset$ and the result is trivial.

\bigskip
\noindent\textit{Induction on $n$.}  
Assume the theorem valid for $\overline{G_{k,m}}$ with $k\le n$, and now take $G=\overline{G_{n+1,m}}$.

\medskip
\noindent\textit{Case $r=2$.}
The edges of $\overline{\con_2(G)}$ containing $x_{1,1}$ include
\[
\{x_{1,1},x_{1,2},x_{1,3}\},\ 
\{x_{1,1},x_{1,2},x_{2,2}\},\ 
\{x_{1,1},x_{1,2},x_{2,1}\},\
\{x_{1,1},x_{2,1},x_{2,2}\},\
\{x_{1,1},x_{2,1},x_{3,1}\}.
\]
Choose
\[
e_1=\{x_{1,1},x_{1,3}\},\quad
e_2=\{x_{1,1},x_{3,1}\},\quad
e_3=\{x_{1,1},x_{2,2}\},\quad
e_4=\{x_{1,1},x_{1,2}\},
\]
and set $\mathcal{C}=\overline{\con_2(G)}\setminus\{e_1,\dots,e_4\}$.
Then $\mathcal{C}$ contains no edge involving $x_{1,1}$.
Next consider the edges of $\mathcal{C}$ containing $x_{1,2}$:
$\{x_{1,2},x_{2,2},x_{2,1}\},\ 
\{x_{1,2},x_{2,2},x_{1,3}\},\ 
\{x_{1,2},x_{2,2},x_{2,3}\},\ 
\{x_{1,2},x_{2,2},x_{3,2}\},\ 
\{x_{1,2},x_{1,3},x_{1,4}\},\ 
\{x_{1,2},x_{1,3},x_{2,3}\}.$
Select
\[
e_5=\{x_{1,2},x_{1,4}\},\quad
e_6=\{x_{1,2},x_{3,2}\},\quad
e_7=\{x_{1,2},x_{2,1}\},\quad
e_8=\{x_{1,2},x_{2,3}\},\quad
e_9=\{x_{1,2},x_{1,3}\}.
\]
Each satisfies
$e_i\in\Sim\!\bigl(\overline{\con_2(G)}\setminus\{e_1,\dots,e_{i-1}\}\bigr),~i=1,\dots,9.$
After deleting $e_1,\dots,e_9$, no edge containing $x_{1,2}$ remains.
Repeating the same procedure for each vertex $x_{1,j}$ produces edges
$e_{10},\dots,e_m$ with
\[
e_i\in\Sim\!\bigl(\overline{\con_2(G)}\setminus\{e_1,\dots,e_{i-1}\}\bigr)
\quad\text{for all } i\le m,
\]
and
$E\bigl(\overline{\con_2(G)}\setminus\{e_1,\dots,e_m\}\bigr)
=
E\bigl(\overline{\con_2(\overline{G_{n,m}})}\bigr).$
The induction hypothesis then implies that $\overline{\con_2(G)}$ is chordal, completing the case.

\smallskip
\noindent\emph{Case $r=3$.}
The edges of $\overline{\con_3(G)}$ containing $x_{1,1}$ are
\[
\{x_{1,1},x_{1,2},x_{2,1},x_{2,2}\},\ 
\{x_{1,1},x_{1,2},x_{1,3},x_{2,2}\},\ 
\{x_{1,1},x_{2,1},x_{2,2},x_{3,1}\}.
\]
Choose
\[
e_1=\{x_{1,1},x_{1,3},x_{2,2}\},\qquad
e_2=\{x_{1,1},x_{2,2},x_{3,1}\},\qquad
e_3=\{x_{1,1},x_{1,2},x_{2,1}\},
\]
and set $\mathcal{D}=\overline{\con_3(G)}\setminus\{e_1,e_2,e_3\}$.
Then $\mathcal{D}$ contains no edge involving $x_{1,1}$.
Next consider the edges of $\mathcal{D}$ containing $x_{1,2}$:
\[
\{x_{1,2},x_{1,3},x_{2,2},x_{2,3}\},\ 
\{x_{1,2},x_{2,1},x_{2,2},x_{2,3}\},\ 
\{x_{1,2},x_{2,1},x_{2,2},x_{3,2}\},
\]
\[
\{x_{1,2},x_{2,2},x_{2,3},x_{3,2}\},\ 
\{x_{1,2},x_{1,3},x_{1,4},x_{2,3}\}.
\]
Select
\[
e_4=\{x_{1,2},x_{2,1},x_{3,2}\},\quad
e_5=\{x_{1,2},x_{3,2},x_{2,3}\},\quad
e_6=\{x_{1,2},x_{2,1},x_{2,3}\},
\]
\[
e_7=\{x_{1,2},x_{1,4},x_{2,3}\},\quad
e_8=\{x_{1,2},x_{1,3},x_{2,3}\}.
\]
After removing $e_1,\dots,e_8$, no remaining edge contains $x_{1,1}$ or $x_{1,2}$.
Repeating the same procedure for the remaining vertices $x_{1,j}$ produces edges
$e_9,\dots,e_m$ such that
\[
e_i \in \Sim\!\bigl(\overline{\con_3(G)}\setminus\{e_1,\dots,e_{i-1}\}\bigr)
\qquad\text{for all } i\le m,
\]
and
$E\bigl(\overline{\con_3(G)}\setminus\{e_1,\dots,e_m\}\bigr)
=
E\bigl(\overline{\con_3(\overline{G_{n,m}})}\bigr).$
By the induction hypothesis, $\overline{\con_3(G)}$ is chordal, completing the case.

\smallskip
\emph{Case $r=4$.}  
Observe that no edge of $\overline{\mathcal C_r(G)}$ contains $x_{1,1}$.
 Edges containing $x_{1,2}$ include  
$\{x_{1,2},x_{2,2},x_{2,3},x_{2,1},x_{3,2}\}$.  
Choose $e_1=\{x_{1,2},x_{2,1},x_{2,3},x_{3,2}\}$, simplicial; iterate as before to reduce to  
$\overline{\con_4(\overline{G_{n,m}})}$.

\smallskip
\emph{Case $r\ge5$.}  
No $r$-edge contains any $x_{1,j}$ with $1\le j\le m$, hence the result follows from the induction hypothesis.

Thus $\con_r(G)$ is co-chordal for all $n,m\ge2$ and $r\ge2$.
\end{proof}

\section{Regularity and Linearity Results}\label{reg-ls}

In this section we obtain regularity bounds and linearity results for ideals associated with several graph families.  
We show that for co-chordal-cactus graphs and co-grid graphs one has $\reg(I_1(G)) \le 3$, and the same bound is known for $(2K_2,C_4)$-free graphs.  
Moreover, for all these classes and every $r \ge 2$, the connected ideals $I_r(G)$ admit linear resolutions.  
We begin with the regularity of edge ideals of co-chordal-cactus graphs.

\begin{theorem}\label{thm:reg-co-cactus}
If $G$ is a co-chordal-cactus graph, then $\reg(I_1(G))\le 3$.
\end{theorem}

\begin{proof}
Let $(T,\{H_e\}_{e\in E(T)})$ be the co-chordal-cactus decomposition of $\overline{G}$, and let
$n = |V(T)|$.  
We prove the statement by induction on $n$.
If $n=2$, then $T$ is a single edge, so $\overline{G}$ is either chordal or a single cycle.  
If $\overline{G}$ is chordal, Fröberg’s theorem implies $\reg(I_1(G))=2$;  
if $\overline{G}$ is a cycle, \cite[Theorem~2.6]{MYZ12} gives $\reg(I_1(G))\le 3$.

Assume $n \ge 3$, and suppose first that $\overline{G}$ has no connector vertex.  
Then each $H_e$ meets all others in at most one vertex, and no vertex lies in more than one component; hence $T$ is a disjoint union of edges.  
Therefore, by \cite[Proposition~3.12]{amir}, together with Fröberg’s theorem and \cite[Theorem~2.6]{MYZ12}, we again obtain
$\reg(I_1(G)) \le 3$.

Now assume that $\overline{G}$ contains a connector vertex $x$.  
Applying the standard short exact sequence (see \cite[Lemma~2.10]{huneke}),
\[
\reg(I_1(G))
\le \max\{\reg(I_1(G):x)+1,\ \reg(I_1(G),x)\},
\]
and using \cite[Lemma~3.2]{huneke},
\[
\reg(I_1(G),x)=\reg(I_1(G\setminus x)),\qquad
\reg(I_1(G):x)=\reg(I_1(G\setminus N_G[x])),
\]
we reduce to smaller co-chordal-cactus graphs.
Deleting the connector vertex $x$ decreases the number of vertices of $T$, so by the induction hypothesis,
$\reg(I_1(G\setminus x)) \le 3.$
Moreover,
$\overline{G\setminus N_G[x]}=\overline{G}[\,N_{\overline{G}}(x)\,]$
is a disjoint union of chordal graphs: each $H_e\setminus\{x\}$ is chordal, and deleting a vertex from a cycle produces a path.  
Thus Fröberg’s theorem implies
$\reg(I_1(G\setminus N_G[x])) = 2.$
Combining the bounds above,
\[
\reg(I_1(G))
\le \max\{\,2+1,\ 3\,\}=3.
\]

This completes the proof.
\end{proof}

We now extend the regularity bound to the class of co-grid graphs.  
As the complement of a grid graph \(G_{n,m} = P_n \square P_m\), this class exhibits similar structural behavior leading to the same upper bound on regularity.

\begin{theorem}\label{reg-grid}
If $G=\overline{G_{n,m}}$, then $\reg(I_1(G))\le 3$.
\end{theorem}

\begin{proof}
Let 
$V(G)=\{x_{i,j}\mid 1\le i\le n,\;1\le j\le m\}$  
and  
$R=\K[x_{i,j}\mid 1\le i\le n,\;1\le j\le m]$.  
Set $I_0=I_1(G)$ and for $1\le i\le m$ define  
$I_i=(I_1(G),x_{1,1},\ldots,x_{1,i})$.  
For $0\le i\le m-1$ the short exact sequence
\[
0\to R/(I_i:x_{1,i+1})(-1)\to R/I_i\to R/I_{i+1}\to 0
\]
yields the standard bound
\begin{equation*}
\reg(R/I_0)\le
\max\Big\{
\reg(R/I_m),\ 
\max_{0\le i\le m-1}\!\big(\reg(R/(I_i:x_{1,i+1}))+1\big)
\Big\}.
\tag{$\ast$}
\end{equation*}

We induct on $m$.  
For $m=1$, $G_{n,1}$ is a path; thus, by \cite[Theorem~1]{froberg},  
$\reg(I_1(\overline{G_{n,1}}))\le 3$.  
Assume the claim holds for all $G_{n,k}$ with $k<m$, and let $G=\overline{G_{n,m}}$.

\medskip
\noindent\textit{(i) The term $\reg(R/I_m)$.}  
The quotient $R/I_m$ is (after relabelling) the edge-ideal quotient of $\overline{G_{n,m-1}}$.  
By induction, $\reg(R/I_m)\le3$.

\medskip
\noindent\textit{(ii) The terms $\reg(R/(I_i:x_{1,i+1}))$.}  
Let $x=x_{1,i+1}$.  
Standard colon computations give
\[
\reg(R/(I_i:x))=\reg\bigl(I_1(G\setminus N_G[x])\bigr).
\]
In $G=\overline{G_{n,m}}$, the closed neighborhood of a first-row vertex $x$ removes a small local strip in the grid; the remaining graph has complement $\overline{G\setminus N_G[x]}$ chordal.  
Thus $G\setminus N_G[x]$ is co-chordal, so by \cite[Theorem 1]{froberg},
$\reg(R/(I_i:x))=1.$

\medskip
Substituting (i) and (ii) into $(\ast)$ gives
$\reg(R/I_0)
\le \max\{\,3,\;\max_i(1+1)\,\}
=2.$
Hence $\reg(I_1(G))=\reg(R/I_0)\le3$, completing the proof.
\end{proof}

Combining the preceding results, we arrive at a unified statement describing the linearity of resolutions for the ideals associated with these graph classes.

\begin{corollary}\label{main:cor}
Let $r \ge 2$.  
Then the ideal $I_r(G)$ has a linear resolution whenever $G$ belongs to one of the following classes:
\begin{enumerate}
    \item $G$ is a co-chordal-cactus graph;
    \item $G$ is $(2K_2, C_4)$-free;
    \item $G$ is the complement of a grid graph $G_{n,m} = P_n \square P_m$.
\end{enumerate}
\end{corollary}

\begin{proof}
The assertion follows directly from Theorems~\ref{thm:main-cochordal}, \ref{2k2-c4-free}, and \ref{thm:con-cochordal},  
together with \cite[Theorem~3.3]{BYZ17}.
\end{proof}

As a consequence of Corollary~\ref{main:cor}, our framework also recovers several previously known results established in~\cite{DRSV24}.

\begin{corollary}[\cite{DRSV24}, Corollaries~3.14 and~5.5]\label{cor:DRSV24}
Let $r \ge 2$ be an integer. Then:
\begin{enumerate}
    \item If $G$ is a co-chordal graph, then the ideal $I_r(G)$ has a linear resolution.
    \item For $n \ge 4$, the ideal $I_r(\overline{C_n})$ also has a linear resolution.
\end{enumerate}
\end{corollary}
In Corollary~\ref{main:cor} it is shown that if $\reg(I_1(G))=2$, equivalently if $G$ is co–chordal, 
then $I_r(G)$ has a linear resolution for all $r \ge 2$.  
This naturally raises the question of whether a similar stabilization phenomenon holds for graphs whose edge ideals have slightly larger regularity.
It is known that if $G$ is a co–chordal-cactus graph, a $(2K_2,C_4)$–free graph, or a co–grid graph, 
then $\reg(I_1(G)) \le 3$ (see Theorem~\ref{thm:reg-co-cactus}, \cite[Proposition 2.11]{Nursel}, and Theorem~\ref{reg-grid}).  
Moreover, by Corollary~\ref{main:cor}, in each of these cases the higher connected ideals $I_r(G)$ are linear for all $r \ge 2$.  
This motivates the following natural question:
\[
\text{If }\reg(I_1(G)) = 3,\ \text{must } I_r(G) \text{ have a linear resolution for all } r \ge 2?
\]

The following example shows that the answer is negative.

\begin{example}
Consider the graph $G$ with edge ideal
\[
I_1(G)
= (x_1x_2,\, x_1x_3,\, x_1x_7,\, x_2x_3,\, x_2x_7,\, x_3x_7,\,
    x_4x_5,\, x_4x_6,\, x_5x_6,\, x_5x_7,\, x_6x_7).
\]
The graph $G$ is chordal and has induced matching number $2$.  
Hence, by~\cite[Corollary 6.9]{ha_adam}, we obtain $\reg(I_1(G)) = 3$.
Now consider the clutter $\con_2(G)$.  
The sets $\{x_1,x_2,x_3\}$ and $\{x_4,x_5,x_6\}$ form an induced matching in $\con_2(G)$.  
By \cite[Theorem 3.7]{BCDMS22}, this implies
$\reg(I_2(G)) \ge 5,$
and consequently $I_2(G)$ does not have a linear resolution.
\end{example}

The preceding discussion also connects to a question raised in~\cite{AJM24}.  
If $G$ is $(2K_2,C_4)$-free or, more generally, $2K_2$–free and $(r+1)$-claw-free, then $I_r(G)$ is known to be linear for all $r \ge 2$ (Theorem~\ref{main:cor} and~\cite[Theorem 6.2]{AJM24}).  
In~\cite[Question~7.1]{AJM24}, the authors asked whether every $2K_2$-free graph satisfies that $I_r(G)$ has a linear resolution for all $r \ge 2$.
This is not the case: the graph exhibited in~\cite[Counterexample~1.10]{nevo_peeva} is $2K_2$-free, yet a computation using \texttt{Macaulay2}~\cite{M2} shows that
$\reg(I_2(G)) = 4,$
so $I_2(G)$ does not have a linear resolution.

These examples lead to the following refined question:

\begin{question}
If $G$ is a $2K_2$-free graph and $\reg(I_1(G)) = 3$, does $I_r(G)$ have a linear resolution for all $r \ge 2$?
\end{question}

\section{Linear Resolutions of Powers of Ideals}\label{main1}

In this section we study the linearity of powers of the ideals $I_r(G)$ arising from graphs.  
We show that $I_r(G)^q$ has a linear resolution whenever $G$ belongs to several significant graph classes.

We begin with the necessary background.  
Let $R=\K[x_1,\ldots,x_n]$.  
A monomial ideal $I\subset R$ generated in a single degree is called \emph{polymatroidal} if it satisfies the following exchange axiom:  
for any $u,v\in\mathcal{G}(I)$ with $\deg_{x_i}(u)>\deg_{x_i}(v)$ for some $i$, there exists a $j$ with $\deg_{x_j}(u)<\deg_{x_j}(v)$ such that  
$x_j\cdot \frac{u}{x_i}\ \in\ \mathcal{G}(I),$
where $\mathcal{G}(I)$ denotes the minimal monomial generating set of~$I$.
A useful generalization, introduced by Kokubo and Hibi~\cite{KK06}, is that of \emph{weakly polymatroidal ideals}.  
Let  
\[
f=x_1^{a_1}\cdots x_n^{a_n},\qquad 
g=x_1^{b_1}\cdots x_n^{b_n}\ \in\ \mathcal{G}(I).
\]
Then $I$ is \emph{weakly polymatroidal} if, whenever
\[
a_1=b_1,\;\ldots,\;a_{t-1}=b_{t-1},\qquad a_t>b_t\quad\text{for some }t,
\]
there exists $\ell>t$ such that
$x_t\cdot\frac{g}{x_\ell}\ \in\ I.$

For any monomial $m\in R=\K[x_1,\ldots,x_n]$, we write
$\Supp(m)=\{\,x_i\mid x_i\text{ divides }m\,\}$
for its \emph{support}.

\begin{lemma}\label{lem:five-divisors}
Let $m=x_1^{a_1}\cdots x_n^{a_n}$ be a monomial of degree $6$ with $|\Supp(m)|\ge5$.  
Then $m\in\mathcal{G}\!\left((I_2(\overline{P_n}))^2\right)$.
\end{lemma}

\begin{proof}
Let $a<b<c<d<e<f$ be the distinct variables dividing $m$, with $f$ present only if $|\Supp(m)|=6$.

\smallskip
\noindent\textit{Case 1: $|\Supp(m)|=5$.}  
Exactly one variable appears with exponent $2$.  
Up to symmetry:
\[
\begin{aligned}
a^2bcde &= (acd)(abe),\\
ab^2cde &= (abd)(bce),\\
abc^2de &= (acd)(bce),
\end{aligned}
\]
and the cases $abcd^2e$, $abcde^2$ are analogous.  
Each factor is a squarefree degree-$3$ monomial corresponding to a connected $3$-vertex induced subgraph of $\overline{P_n}$; hence each lies in $\mathcal{G}(I_2(\overline{P_n}))$.  
Thus $m\in\mathcal{G}\left((I_2(\overline{P_n}))^2\right)$.

\smallskip
\noindent\textit{Case 2: $|\Supp(m)|=6$.}  
Then $m=abcdef$ and
$m=(acd)(bef),$
with both factors squarefree of degree $3$ and supported on connected $3$-vertex induced subgraphs of $\overline{P_n}$; hence both are generators of $I_2(\overline{P_n})$.

\smallskip
In all cases $m$ factors as a product of two generators of $I_2(\overline{P_n})$, proving the claim.
\end{proof}

We are now ready to state the main result of this section, which establishes the existence of 
linear resolutions for all powers of $I_r(G)$ when $G$ is the complement of a tree with bounded degree.

\begin{theorem}\label{tree-complement}
Let $T$ be a tree with maximum degree $\Delta(T)\le r$, where 
$\Delta(T):=\max\{\deg_T(v)\mid v\in V(T)\}$, 
and let $G=\overline{T}$. 
Then for every $q\ge1$ and $r\ge2$, the ideal $I_r(G)^q$ has a linear resolution.
\end{theorem}

\begin{proof}
It suffices to show:
\[
\text{(i) } (I_2(\overline{P_n}))^q\ \text{is weakly polymatroidal for } r=2,\qquad
\text{(ii) } I_r(G)\ \text{is polymatroidal for } r\ge3.
\]
Then, by \cite{KK06} and \cite[Corollary 12.6.4]{Herzog'sBook}, all powers have linear resolutions.

\medskip
\noindent\textit{Case: $r=2$.}  
If $\Delta(T)=2$, then $T\cong P_n$ with vertices $x_1<\dots<x_n$.  
For a monomial $m$, let $\Supp(m)=\{x_i: x_i\mid m\}$.  
Let
$u=x_1^{a_1}\cdots x_n^{a_n}$,
$v=x_1^{b_1}\cdots x_n^{b_n}\in\mathcal G\big((I_2(\overline{P_n}))^q\big),$
with
\[
a_1=b_1,\dots,a_{t-1}=b_{t-1},\qquad a_t>b_t.
\]
We must find $j>t$ such that 
$x_t\cdot \frac{v}{x_j}\ \in\ (I_2(\overline{P_n}))^q .$
There exists some $j>t$ with $x_j\mid v$; otherwise $\Supp(v)\subseteq \{x_1,\dots,x_t\}$, and since $a_i=b_i$ for $i<t$ while $a_t>b_t$, we obtain
$\deg(u)\ge \deg(v)+1,$
contradicting $\deg(u)=\deg(v)=q(r+1)$.  
Choose such a $j>t$ with $\dd_{P_n}(x_t,x_j)$ minimal.  
Let $m\mid v$ be a monomial divisor of $v$ with $x_j\mid m$, and write
$m = a b x_j.$

\smallskip
\textit{(i) $x_t\nmid m$.}  
Then
$x_t\frac{m}{x_j}=abx_t.$
If $abx_t\in \mathcal{G}(I_2(\overline{P_n}))$ we are done.  
Otherwise $P_n[\{a,b,x_t\}] \cong P_n[\{x_{t-2},x_{t-1},x_t\}]$, which forces $q\ge2$;  
if $q=1$ then $x_{t-2}x_{t-1}x_t=u$, a contradiction.
If there exists a monomial $m'\mid v$ with
$|\Supp(m')\cap\{a,b,x_t\}|=1,$
then Lemma~\ref{lem:five-divisors} completes the argument.  
Thus we may assume that
\[
|\Supp(m')\cap\{a,b,x_t\}| \ge 2
\qquad\text{for all } m'\mid v,\ m'\ne abx_j.
\]

Since $a,b < x_t$, we also have $a,b \mid u$.  
Because $a_t > b_t$, the presence of $x_t$ in $v$ forces its presence in $u$; therefore
\[
\sum_{x\in\{a,b,x_t\}}\deg_x(u)
>
\sum_{x\in\{a,b,x_t\}}\deg_x(v).
\]
Explicitly,
$\sum_{x\in\{a,b,x_t\}}\deg_x(v)=2q.$
If every monomial $m_0\mid u$ satisfies 
$|\Supp(m_0)\cap\{a,b,x_t\}|=2,$
then the same sum for $u$ equals $2q$, contradicting the strict inequality above.
If there exists $m_0\mid u$ with 
$|\Supp(m_0)\cap\{a,b,x_t\}|=1$
and
$|\Supp(m)\cap\{a,b,x_t\}|=2
\quad\text{for all } m\mid u,\ m\ne m_0,$
then
\[
\sum_{x\in\{a,b,x_t\}}\deg_x(u)
=
2(q-1)+1
<
2q
=
\sum_{x\in\{a,b,x_t\}}\deg_x(v),
\]
again a contradiction.  
A similar contradiction arises whenever some $m_0\mid u$ satisfies   
$|\Supp(m_0)\cap\{a,b,x_t\}|=1$.
Hence we must have 
$|\Supp(m)\cap\{a,b,x_t\}| \ge 2~\text{for all } m\mid u.$
Now assume 
$|\Supp(m)\cap\{a,b,x_t\}| = 2
~\text{for all } m\mid u.$
Then 
\[
\sum_{x\in\{a,b,x_t\}}\deg_x(u)
=
\sum_{x\in\{a,b,x_t\}}\deg_x(v),
\]
contradicting the earlier strict inequality.
Therefore $abx_t$ must be a monomial divisor of $m$, contradicting  
$abx_t \notin \mathcal{G}(I_2(\overline{P_n}))$.  
Thus such an $m'$ must exist, and Lemma~\ref{lem:five-divisors}
implies that
$x_t(v/x_j) \in (I_2(\overline{P_n}))^q.$

\smallskip
\textit{(ii) $x_t \mid m$.}  
Then $m = a x_t x_j$, and hence $x_t \frac{m}{x_j} = a x_t^2$.
In this case, there exists a monomial $m' \mid v$ with $x_t \nmid m'$.  
Let
$\{m_1,\ldots,m_n\} = \{\, m : m \mid v \ \text{and}\ x_t \nmid m \,\}.$
If there exists a variable $x_{t'} > x_t$ dividing some $m_j$, then we reduce to \textit{Case~(i)}.  
Thus we may assume that every variable $y$ dividing any $m_j$ satisfies $y < x_t$.
If $|\Supp(m')\cap\{a,x_t\}|=0$, we are done.  
Assume instead that 
$\Supp(m')\cap\{a,x_t\}=\{a\},$
so $m' = a b c$ for some $b,c < x_t$.
Then
$m' \cdot a x_t^2 = (a x_t b)(a x_t c).$
Since $|N_{P_n}(x_t)| = 2$, at least one of the monomials 
$a x_t b$ or $a x_t c$ lies in $\mathcal{G}(I_2(\overline{P_n}))$.  
If both do, we are done.  
Otherwise $a x_t b \notin \mathcal{G}(I_2(\overline{P_n}))$, and thus  
$P_n[\{a,b,x_t\}] \cong P_n[\{x_{t-2}, x_{t-1}, x_t\}],$
which implies $q \ge 3$.
As in \textit{Case~(i)}, a contribution-counting argument now yields a monomial  
$m'' \mid v$ satisfying
$|\Supp(m'')\cap\{a,b,x_t\}| = 1.$
Lemma~\ref{lem:five-divisors} then shows that $x_t (v/x_j) \in (I_2(\overline{P_n}))^q$.

Thus $(I_2(\overline{P_n}))^q$ is weakly polymatroidal in \textit{Case~(ii)} as well.

\medskip
\noindent\textit{Case:  $r \ge 3$.}  
Let
$u = x_{i_1}\cdots x_{i_{r+1}}$, 
$v = x_{j_1}\cdots x_{j_{r+1}}$
be distinct minimal generators of $I_r(G)$, and set
$U = \Supp(u)$, $V = \Supp(v).$
If $|U \setminus V| = 1$, the exchange property is immediate.  
Assume now that $|U \setminus V| \ge 2$; then also $|V \setminus U| \ge 2$.  
Choose
\[
x_{i_p} \in U \setminus V, \qquad  
x_{j_t}, x_{j_q} \in V \setminus U.
\]

Consider the induced subgraph
$G\big[\{x_{j_t}\} \cup (U \setminus \{x_{i_p}\})\big].$
If it is connected, the exchange property follows.  
Otherwise, its complement
$\overline{G}\big[\{x_{j_t}\} \cup (U \setminus \{x_{i_p}\})\big]$
is a connected induced subgraph of the tree $T$.  
By Lemma~\ref{join-dis}, this graph has the structure
$B_1 \vee B_2 \vee \cdots \vee B_k.$
If $k\ge 3$, then $\overline{G}$ would contain a cycle, contradicting the fact that $\overline{G}$ is a tree. 
Thus $k=2$.
Moreover, if both $B_1$ and $B_2$ had at least two vertices, a cycle would again occur.  
Thus, one block is a singleton; assume $|B_1| = 1$.  
Since any edge inside $B_2$ would also create a cycle, $B_2$ is an independent set.  
Therefore,
$\overline{G}\big[\{x_{j_t}\} \cup (U \setminus \{x_{i_p}\})\big] \cong \text{a star graph}.$

\smallskip
\emph{(i) $x_{j_t}$ is the central vertex.}  
If $x_{j_q}$ is adjacent in $\overline{G}$ to no vertex of $U \setminus \{x_{i_p}\}$, then
$G\big[\{x_{j_q}\} \cup (U \setminus \{x_{i_p}\})\big]$
is connected.  
Otherwise, suppose $x_{j_q}$ is adjacent in $\overline{G}$ to exactly one vertex
$v' \in U \setminus \{x_{i_p}\}$.  
Since $r \ge 3$, choose
$z \in U \setminus \{x_{i_p}, v'\}$
that is not adjacent to $x_{j_q}$ in $\overline{G}$.  
In $G$, the set $U \setminus \{x_{i_p}\}$ forms a clique, and the presence of $z$ ensures that $x_{j_q}$ remains connected to this clique.  

\smallskip
\emph{(ii) $x_{j_t}$ is a non-central vertex.}  
Let $c \in U \setminus \{x_{i_p}\}$ be the center of the star.  
If $x_{j_q}$ is adjacent in $\overline{G}$ to no vertex of $U \setminus \{x_{i_p}\}$, then
$G\big[\{x_{j_q}\} \cup (U \setminus \{x_{i_p}\})\big]$
is connected.  
Otherwise, suppose $x_{j_q}$ is adjacent in $\overline{G}$ to exactly one vertex
$v' \in U \setminus \{c, x_{i_p}\}.$
Since $r \ge 3$, choose
$z \in U \setminus \{x_{i_p}, c, v'\}$
that is not adjacent to $x_{j_q}$ in $\overline{G}$.  
In $G$, the set $U \setminus \{x_{i_p}, c\}$ is a clique, and since $x_{j_q}$ is connected to both $z$ and $c$, the induced subgraph is connected.

In all subcases, the exchange property holds.  
Thus $I_r(G)$ is polymatroidal for every $r \ge 3$.
\end{proof}

Let $n_1,\ldots,n_p$ be positive integers.  
The graph $K_{n_1,\ldots,n_p}$, called the \emph{complete $p$-partite graph}, is the
graph whose vertex set admits a partition
$V_1 \sqcup \cdots \sqcup V_p$ with $|V_i| = n_i$, such that two vertices are
adjacent precisely when they belong to different parts.

We next introduce a family of graphs obtained by gluing several complete graphs along a common core.  
Let
$V = \{x_1,\ldots,x_p\} \;\cup\; \{\, y_{ij} \mid 1\le i\le n,\; 1\le j\le m_i \,\},$
where $p\ge1$ and $m_i \ge 1$ for all $i$.  
For each $1\le i\le n$, let $G^{p,m_i}$ denote the complete graph on the set
\[
V(G^{p,m_i}) = \{x_1,\ldots,x_p,\, y_{i1},\ldots,y_{im_i}\}.
\]
The graph $\Gamma_{p,m_1,\ldots,m_n}$ is defined by
\[
V(\Gamma_{p,m_1,\ldots,m_n}) = V, \qquad
E(\Gamma_{p,m_1,\ldots,m_n}) = \bigcup_{i=1}^n E(G^{p,m_i}).
\]
Thus $\Gamma_{p,m_1,\ldots,m_n}$ is obtained by taking $n$ cliques
$G^{p,m_i}$, each containing the common core $\{x_1,\ldots,x_p\}$,
and adjoining for each $i$ the distinct peripheral vertices
$\{y_{i1},\ldots,y_{im_i}\}$.

\smallskip

In addition to complements of trees, several further classes of graphs admit the persistence of linear
resolutions for all powers of their $r$-connected ideals.  
The following result collects these cases.

\begin{theorem}\label{thm:linear-resolution-cases}
For every integer $q \ge 1$, the ideal $I_r(G)^q$ has a linear resolution in each of the following cases:
\begin{enumerate}
    \item $G = K_{n_1,\ldots,n_p}$, for all $r \ge 1$;
    \item $G = \overline{C_n}$, for all $r \ge 3$;
    \item $G = \Gamma_{p,m_1,\ldots,m_n}$ with $m_i \le r$ for all $i$, for all $r \ge 1$.
\end{enumerate}
\end{theorem}

\begin{proof}
We verify that $I_r(G)$ satisfies the exchange property, hence is polymatroidal; therefore $I_r(G)^q$ has a linear resolution for all $q \ge 1$ by~\cite[Corollary 12.6.4]{Herzog'sBook}.

\medskip
\noindent\textbf{(1) $G = K_{n_1,\ldots,n_p}$.}
Let $u,v \in \mathcal{G}(I_r(G))$ be distinct generators.  
Pick $x_{i_p} \in \Supp(u)\setminus\Supp(v)$ and set
$S := \Supp(u)\setminus\{x_{i_p}\}$,  $|S|=r.$
We must find $x_{j_w} \in \Supp(v)\setminus\Supp(u)$ such that $G[S\cup\{x_{j_w}\}]$ is connected.

\smallskip
\textit{Observation.}  
If a set $T$ with $|T|=r+1$ satisfies $T\cap V_a\neq\emptyset$ and $T\cap V_b\neq\emptyset$ for some $a\neq b$, then $G[T]$ is connected.  
We use this repeatedly.

\smallskip
\textit{Case 1: $S$ meets at least two partite sets.}  
Then for any $x_{j_w}\in\Supp(v)\setminus\Supp(u)$, the set $S\cup\{x_{j_w}\}$ still meets at least two parts, hence $G[S\cup\{x_{j_w}\}]$ is connected.  

\smallskip
\textit{Case 2: $S \subseteq V_1$.}  
Since $\Supp(u)$ induces a connected graph, $x_{i_p}\notin V_1$; say $x_{i_p}\in V_2$.  
If $\Supp(v)\setminus\Supp(u) \subseteq V_1$, then $\Supp(v)\subseteq V_1$, contradicting that $G[\Supp(v)]$ is connected (as $V_1$ is independent).  
Thus there exists $x_{j_w}\in(\Supp(v)\setminus\Supp(u))\cap V_k$ for some $k\neq 1$.  
Hence $S\cup\{x_{j_w}\}$ meets $V_1$ and $V_k$, so $G[S\cup\{x_{j_w}\}]$ is connected.

\medskip
In all cases, a suitable $x_{j_w}$ exists; thus $I_r(G)$ satisfies the exchange property, and hence is polymatroidal.

\vskip1mm
\noindent\textbf{(2)} 
Observe that for $n \le 4$, we have $I_r(G) = (0)$. So clearly $I_r(G)^q$ has a linear resolution.
Assume now that $n \ge 5$.
Let 
$u=x_{i_1}\cdots x_{i_{r+1}},\, v=x_{j_1}\cdots x_{j_{r+1}}\in\G(I_r(G))$
be distinct generators.  
Choose $x_{i_p}\mid u$, $x_{i_p}\nmid v$, and $x_{j_w}\mid v$, $x_{j_w}\nmid u$.  
We show that
\[
G\big[\{x_{j_w},x_{i_1},\dots,x_{i_{p-1}},x_{i_{p+1}},\dots,x_{i_{r+1}}\}\big]
\]
is connected, proving the exchange property.  
For vertices $x,y$, we write $x\nsim_{C_n} y$ to mean $\{x,y\}\notin E(C_n)$.  
Since each vertex of $C_n$ has degree $2$, write 
$N_{C_n}(x_{j_w})=\{a,b\}$.

\smallskip
\noindent\textit{Case 1.}  
$a,b\notin \{x_{i_1},\dots,x_{i_{p-1}},x_{i_{p+1}},\dots,x_{i_{r+1}}\}$.  
Then $x_{j_w}$ has no forbidden adjacency inside this set, hence the induced subgraph of $G$ is connected.

\smallskip
\noindent\textit{Case 2.}  
$a\in \{x_{i_1},\dots,x_{i_{p-1}},x_{i_{p+1}},\dots,x_{i_{r+1}}\}$ and $b\notin$ this set.  
Then $\exists\,x\in\Supp(u)\setminus\{x_{i_p}\}$ such that $x\nsim_{C_n} x_{j_w}$.

\emph{Subcase 2.1.}  
If $x\nsim_{C_n} a$, then $G[\{a,x,x_{j_w}\}]$ is connected, so the exchange holds.

\emph{Subcase 2.2.}  
If $x \sim_{C_n} a$, then since $r \ge 3$ there exists 
$z \in \Supp(u)\setminus\{a,x,x_{i_p}\}$ with $z \nsim_{C_n} a$.  
Then the induced subgraph $G[\{z,a,x_{j_w}\}]$ is connected.  
Note that $z \nsim_{C_n} x_{j_w}$; otherwise $z=b$, a contradiction.  
Hence connectivity in $G$ is preserved.

\smallskip
\noindent\textit{Case 3.}  
$a,b\in \{x_{i_1},\dots,x_{i_{p-1}},x_{i_{p+1}},\dots,x_{i_{r+1}}\}$.  
Then $\exists\,z\in\Supp(u)$ with $z\nsim_{C_n} a$ or $z\nsim_{C_n} b$.  
Assume $z\nsim_{C_n} a$; then 
$\{b,a,z,x_{j_w}\}$ induces a connected subgraph in $G$.

\smallskip
In every case, the induced subgraph containing $x_{j_w}$ and the $r$ remaining vertices of $u$ is connected in $G$, proving the exchange property.  
Thus $I_r(G)$ is polymatroidal for all $r\ge3$.

\vskip1mm
\noindent\textbf{(3) $G=\Gamma_{p,m_1,\dots,m_n}$.}
Let
$u=x_{i_1}\cdots x_{i_{r+1}},\;
v=x_{j_1}\cdots x_{j_{r+1}}\in\G(I_r(G))$
be distinct generators.  
Choose $x_{i_p}\mid u$, $x_{i_p}\nmid v$.  
We must find $x_{j_w}\mid v$  and $x_{j_w}\nmid u$ such that 
\[
x_{j_w}\,\frac{u}{x_{i_p}}\in\G(I_r(G))
\quad\Longleftrightarrow\quad
(\Supp(u)\!\setminus\!\{x_{i_p}\})\cup\{x_{j_w}\}
\text{ induces a connected subgraph of }G.
\]

We distinguish whether $x_{j_w}$ lies in the central clique $K_p$.

\smallskip
\noindent\textit{Case 1.} $x_{j_w}\in K_p$.  
Since every vertex of $K_p$ is adjacent to all vertices of $G$, connectivity of
$(\Supp(u)\!\setminus\!\{x_{i_p}\})\cup\{x_{j_w}\}$ is immediate.

\smallskip
\noindent\textit{Case 2.} $x_{j_w}\notin K_p$.  
Every connected $(r{+}1)$-set intersects $K_p$; hence $v$ contains a vertex
$y\in K_p$.  
If some $y\in K_p$ satisfies $y\mid v$ and $y\nmid u$, then take $x_{j_w}=y$ and reduce to Case~1.  
Thus assume:
\[
(\ast)\qquad K_p\cap\Supp(v)\subseteq\Supp(u).
\]
Hence $\Supp(u)$ also contains a vertex of $K_p$.  We now split according as $x_{i_p}\in K_p$.

\smallskip
\emph{Subcase 2.1.} If $x_{i_p}\notin K_p$, then since $u$ is a connected $(r+1)$-set, there exists 
$z \in \Supp(u)\cap K_p$ with $z \ne x_{i_p}$.  
Because $z$ is adjacent to every vertex of $G$, in particular to all of 
$\Supp(u)\setminus\{x_{i_p},z\}$ and to $x_{j_w}$, the induced subgraph on 
$(\Supp(u)\setminus\{x_{i_p}\}) \cup \{x_{j_w}\}$
remains connected.

\smallskip
\emph{Subcase 2.2.}If $x_{i_p}\in K_p$, then there exists $y\in K_p$ with $y \mid v$.  
By $(\ast)$ we have $y\in\Supp(u)$, and since $x_{i_p}\nmid v$, it follows that $y\ne x_{i_p}$.  
Because $y$ is adjacent to every vertex of $G$, in particular to all of 
$\Supp(u)\setminus\{x_{i_p},y\}$ and to $x_{j_w}$, the induced subgraph on
$(\Supp(u)\setminus\{x_{i_p}\}) \cup \{x_{j_w}\}$
is connected.

\smallskip
In all cases we find $x_{j_w}\mid v$ such that replacing $x_{i_p}$ by $x_{j_w}$ preserves connectivity.  
Hence the exchange property holds for $\G(I_r(G))$, and thus $I_r(G)$ is polymatroidal.
\end{proof}

We introduce next a subclass of split graphs that furnishes an additional family
for which all powers of $I_r(G)$ have linear resolutions.  Recall that a graph is
\emph{split} if its vertex set decomposes as the disjoint union of a clique and
an independent set.  The definition below records a natural restriction obtained
by controlling the adjacency between these two parts.

\begin{definition}\label{def:partially-split}
A graph $G$ is called a \emph{partially split graph} if its vertex set admits a
partition $V(G)=K\sqcup S$ such that $G[K]$ is a clique, $G[S]$ is an independent
set, and
\[
E(G)=E(G[K])\;\cup\;
\{\{x,y\}\mid x\in K',\ y\in S\}
\]
for some nonempty subset $K'\subseteq K$.
\end{definition}

\begin{theorem}\label{partially-split}
Let $G$ be a partially split graph. 
Then for every $r\ge2$ and $q\ge1$, the ideal $I_r(G)^q$ has a linear resolution.
\end{theorem}

\begin{proof}
Let $V(G)=K\sqcup S$ be the partition from Definition~\ref{def:partially-split}, where
$G[K]$ is a clique, $G[S]$ is independent, and all $K$–$S$ edges are incident to a fixed nonempty proper subset $K'\subsetneq K$.  
Fix a variable order
\[
K'=\{x_1< \cdots <x_m\}<K\!\setminus\!K'=\{x_{m+1}<\cdots<x_t\}<S=\{x_{t+1}<\cdots<x_n\}
\]
and let $<_{\mathrm{lex}}$ be the induced lex order.
Let $u=\prod x_i^{a_i}$, $v=\prod x_i^{b_i}$ be distinct minimal generators of $I_r(G)^q$ with
\[
a_1=b_1,\dots,a_{t-1}=b_{t-1},\qquad a_t>b_t.
\]
Each of $u$ and $v$ is the product of $q$ generators of $I_r(G)$ (connected $(r+1)$-sets).  Write these for $v$ as
$m_1,\dots,m_q\in\mathcal{G}(I_r(G))$.

Now there exists $j>t$ with $x_j\mid v$.
Hence choose $j>t$ with $x_j\mid v$, and choose $m\mid v$ with $m\in\mathcal{G}(I_r(G))$ and $x_j\mid m$.
Since $a_t>b_t$, some generator $m'\mid v$ satisfies $x_t\nmid m'$; write
$m'=x_{j_1}\cdots x_{j_{r+1}}$.
We now treat cases based on where the variables dividing $v$ lie.

\medskip
\noindent\textit{Case 1.  $\Supp(v)\subseteq K$.}  
Thus every divisor of $v$ lies in the clique $K$.

\smallskip
\emph{(1.1)  $x_t \nmid m$.}  
Since $x_t \le x_j$, we have $x_t \in K$.  
Therefore,
\[
x_t \cdot \frac{m}{x_j} \in \mathcal{G}(I_r(G))
\qquad\Longrightarrow\qquad
x_t \cdot \frac{v}{x_j} \in (I_r(G))^q,
\]
and the exchange property follows.

\smallskip
\emph{(1.2) $x_t\mid m$.}  
Write
$m=x_{i_1}\cdots x_{i_{r+1}},~ x_{i_p}=x_t,\ x_{i_q}=x_j.$
If $x_j\mid m'$, we reduce to (1.1).  Thus assume $x_j\nmid m'$.  
Since $m\neq m'$, there exists a variable $x_{j_1}\mid m'$ with $x_{j_1}\nmid m$.  
Because $K$ is a clique,
\[
x_t x_{j_2}\cdots x_{j_{r+1}}\in\mathcal{G}(I_r(G))
\quad\text{and}\quad
x_{i_1}\cdots x_{i_{q-1}}x_{j_1}x_{i_{q+1}}\cdots x_{i_{r+1}}\in\mathcal{G}(I_r(G)),
\]
and hence
$x_t\cdot(v/x_j)\in(I_r(G))^q$.

\medskip
\noindent\textit{Case 2.} There exists $y\in S$ with $y\mid v$.  
Write $y=x_j$ and let 
$m=x_{i_1}\cdots x_{i_{r+1}}$.  
Note that $m$ and $m'$ may coincide.

\medskip
\noindent\textit{Subcase 2.1.} $x_t\in K'$.

\begin{itemize}
\item[(A)] If $x_t\nmid m$, then 
$x_t\cdot(m/x_j)\in\mathcal{G}(I_r(G))$, hence  
$x_t\cdot(v/x_j)\in(I_r(G))^q$.

\item[(B)] If $x_t\mid m$, then since $m\neq m'$ we may choose
$x_{j_1}\mid m'$ with $x_{j_1}\nmid m$.  
As $K'$ is universal in $G$,  
$x_t x_{j_2}\cdots x_{j_{r+1}}\in\mathcal{G}(I_r(G))$ and  
$x_{j_1}x_{i_1}\cdots x_t\cdots x_{i_{r+1}}\in\mathcal{G}(I_r(G))$.   
Thus $x_t(v/x_j)\in (I_r(G))^q$.
\end{itemize}

\medskip
\noindent\textit{Subcase 2.2.} $x_t\in K\setminus K'$.

\begin{itemize}
\item[(A)] If $x_t\nmid m$, choose $z\in K'$ with $z\mid m$.  
Then $x_t(m/x_j)\in\mathcal{G}(I_r(G))$ and hence  
$x_t(v/x_j)\in(I_r(G))^q$.

\item[(B)] Assume $x_t\mid m$.  
Choose $z\in K'$ dividing $m$.  
If $x_j\mid m'$, then $x_t(m'/x_j)\in \mathcal{G}(I_r(G))$.  
Thus suppose $x_j\nmid m'$.  
If some $x_{j'}>x_t$ with $x_{j'}\mid m'$, then we are done.  
Hence all variables of $m'$ lie $<x_t$. Observe that $x_j\nmid m'$ and $x_t\nmid m'$ but $x_t,x_j\mid m$.
So, there exist at least two variables  
$y_1,y_2\mid m'$ with $y_i\nmid m$.
Without loss of generality, let $x_{j_1}=y_1$, $x_{j_2}=y_2$.  
Since all $y_i\in K$

\[x_{i_1}\cdots x_{i_t}\cdots x_{i_{j-1}}y_2x_{i_{j+1}}\cdots x_{i_{r+1}},\qquad
x_{i_t}y_1 x_{j_3}\cdots x_{j_{r+1}}
\in\mathcal{G}(I_r(G))
\text{ where } x_{i_t}=x_t.\]
Hence $x_t(v/x_j)\in(I_r(G))^q$.
\end{itemize}

\medskip
\noindent\textit{Subcase 2.3.} $x_t\in S$.  
Thus $j>t$ implies $x_j\in S$.  
Since $m\in\mathcal{G}(I_r(G))$ and $S$ is independent,  
$m$ contains some $z'\in K'$ (the unique class connecting to $S$).

\begin{itemize}
\item[(A)] If $x_t\nmid m$, we immediately obtain the exchange  
$x_t(m/x_j)\in\mathcal{G}(I_r(G))$.

\item[(B)] Assume $x_t\mid m$.  
Write  
$m=x_{i_1}\cdots x_{i_t}\cdots x_{i_j}\cdots x_{i_{r+1}}$,
$x_t=x_{i_t},\ x_j=x_{i_j}.$

If $x_j \mid m'$, then by 2.3(A) we are done.  
Thus assume $x_j \nmid m'$.  
In this case,
\[
x_t \cdot \frac{m}{x_j}
    = x_{i_1}\cdots x_{i_{t-1}}\, x_t^{\,2}\, x_{i_{t+1}}\cdots x_{i_{r+1}}.
\]
If $m'$ contains some variable $x \in K'$, then we may exchange the factor $x_t$
with a variable $y \mid m'$ satisfying $y \nmid m$, yielding the desired exchange.  
Similarly, if $m'$ contains a vertex $y' \in S$, then we are again done, since in that case $m'$ must also contain some $x' \in K'$ with $x' \mid m'$.

Suppose no such $m'$ exists.  
Set
$\{m_{i_1},\dots,m_{i_t}\}
    = \{\, m\in \mathcal{G}(I_r(G)) : m \mid v \ \text{and}\ x_t \nmid m \,\}.$
By the above argument, we may assume that every variable dividing each $m_{i_s}$
lies in $K \setminus K'$.  
Now suppose that $m$ contains another vertex $z_1 \in K'$ with $z_1 \neq z'$.  
Then $x_t, z_1 \nmid m'$, and hence there exist variables $y_1, y_2$ such that  
$y_1, y_2 \mid m'$ but $y_1, y_2 \nmid m$.  
In this situation we may interchange the pair $(x_t, z')$ with $(y_1, y_2)$, yielding the desired exchange.

If no such $z_1$ exists, consider
$\{m_1',\dots,m_b'\}
    = \{\, m \in \mathcal{G}(I_r(G)) : m \mid u \ \text{and}\ x_t \mid m \,\}.$
Then there exist variables $x', y' \in K'$ such that  
$x' \mid m_i'$ and $y' \mid m_j'$ for suitable indices $i, j$.  
Hence some generator $m'' \mid v$ contains both $x'$ and $y'$, say
$m'' = x' y' x_{s_3} \cdots x_{s_{r+1}}.$
Suppose $m \neq m''$.  
Since $x_t \mid m''$ but $x_t \nmid m'$, and $x', y' \nmid m'$,  
there exist variables $w_1, w_2, w_3$ such that  
$w_k \mid m'$ but $w_k \nmid m''$ for $k = 1,2,3$.  
Choose $x_{j_y}$ with $x_{j_y} \mid m'$ and $x_{j_y} \nmid m$.

Define
\[
m_{11}
= x_{i_1}\cdots x_{i_t}\cdots z' \cdots 
   x_{i_{j-1}} x_{j_y} x_{i_{j+1}} \cdots x_{i_{r+1}},
\qquad
m_{12}
= x_{j_1} w_2 \cdots x_{j_{y-1}} x_t x_{j_{y+1}} \cdots x_{j_{r+1}}.
\]
Since $w_2 \nmid m''$ but $w_2 \mid m'$, interchanging $w_2$ and $x'$ yields
\[
m_{13}
= x_{j_1} x' \cdots x_{j_{y-1}} x_t x_{j_{y+1}} \cdots x_{j_{r+1}},
\qquad
m_{14}
= w_2 y' x_{s_3} \cdots x_{s_{r+1}}.
\]
All of $m_{11}, m_{13}, m_{14}$ belong to $\mathcal{G}(I_r(G))$, and thus
$x_t (v/x_j) \in (I_r(G))^q.$

Finally, suppose $m = m''$.  
Then $x', x_t \nmid m'$, and hence there exist variables $y_1, y_2$ such that  
$y_1, y_2 \mid m'$ but $y_1, y_2 \nmid m$.  
In this situation we may interchange the pairs $(y_1, y_2)$ and $(x_t, x')$,  
yielding the desired exchange.
\end{itemize}

In every subcase, an exchange variable is found, verifying the weakly polymatroidal condition.  
Therefore, by \cite{KK06}, each power $I_r(G)^q$ has a linear resolution.
\end{proof}

\vspace*{2mm}
\noindent
\textbf{Acknowledgments.}  
The second author acknowledges support from the Science and Engineering Research Board (SERB)
and the National Board for Higher Mathematics (NBHM).

\vspace*{1mm}
\noindent
\textbf{Data availability statement.}  
Data sharing is not applicable to this article as no datasets were generated or analyzed during the current study.

\vspace*{1mm}
\noindent
\textbf{Conflict of interest.}  
The authors declare that they have no known competing financial interests or personal relationships that could have appeared to influence the work reported in this paper.

\bibliographystyle{abbrv}
\bibliography{refs_reg} 
\end{document}